\documentclass{article}
\usepackage{graphicx}
\usepackage{pgfplots}
\usepackage{subfigure}
\usepackage{caption}
\usepackage[justification=centering]{caption}
\usepackage{amsmath}
\usepackage{amssymb}
\usepackage{amsthm}
\usepackage{listings}
\usepackage{epstopdf}
\usepackage{float}
\usepackage[linesnumbered,ruled]{algorithm2e}
\usepackage[utf8]{inputenc}
\usepackage{enumitem}
\usepackage[top=2cm, bottom=2cm, left=2cm, right=2cm]{geometry}
\usepackage{multicol}

\newtheorem{theorem}{Theorem}

\newtheorem{lemma}{Lemma}
\usepackage[titletoc,title]{appendix}
\usepackage{array}
\usepackage{authblk}

\usepackage{cite}

\makeatletter
\renewcommand{\@biblabel}[1]{#1.}
\makeatother

\begin{document}
	
\title{\color{red}\textbf{Finding Dynamics for Fractals}}
\author[1]{\textbf{Marat Akhmet}\thanks{Corresponding Author Tel.: +90 312 210 5355, Fax: +90 312 210 2972, E-mail: marat@metu.edu.tr}$^{,}$}
\author[2]{\textbf{Mehmet Onur Fen}}
\author[1]{\textbf{Ejaily Milad Alejaily}}
\affil[1]{\textbf{Department of Mathematics, Middle East Technical University, 06800 Ankara, Turkey}}
\affil[2]{\textbf{Department of Mathematics, TED University, 06420 Ankara, Turkey}}

\date{}
\maketitle
	
\noindent \textbf{Abstract.}
The famous Laplace's Demon is not only of strict physical determinism, but also related to the power of differential equations. When deterministically extended structures are taken into consideration, it is admissible that fractals are dense both in the nature and in the dynamics. In particular, this is true because fractal structures are closely related to chaos. This implies that dynamics have to be an instrument of the extension. Oppositely, one can animate the arguments for the Demon if dynamics will be investigated with fractals. To make advances in the direction, first of all, one should consider fractals as states of dynamics. In other words, instead of single points and finite/infinite dimensional vectors, fractals should be points of trajectories as well as trajectories themselves. If one realizes this approach, fractals will be proved to be dense in the universe, since modeling the real world is based on differential equations and their developments. Our main goal is to initiate the involvement of fractals as states of dynamical systems, and in the first step we answer the simple question ``How can fractals be mapped?''. In the present study Julia and Mandelbrot sets are considered as initial points for the trajectories of the dynamics. 
   

\begin{multicols}{2}
    
\section*{\normalsize \color{red}INTRODUCTION}
French mathematicians P. Fatou and G. Julia invented a special iteration in the complex plane \textit{\cite{Julia,Fatou}} such that new  geometrical objects  with unusual properties can be built. One of the famous is the Julia set. Later B. Mandelbrot suggested to call them fractals \textit{\cite{Mandelbrot0}}. Famous fractals appeared as self-similar objects were introduced by Georg Cantor in 1883, and Koch snowflake was discovered in 1904.  

Fractals are in the forefront of researches in many areas of science as well as for interdisciplinary investigations \textit{\cite{Barnsley1,Mitchell}}. One cannot say that motion is a strange concept for fractals. Dynamics are  beside the fractals immediately as they are constructed by iterations. It is mentioned in the book \textit{\cite{Peitgen04}} that it is inadequate to talk about fractals while ignoring the dynamical processes which created them. That is, iterations are in the basis of any fractal, but we still cannot say that differential equations are widely interrelated to fractals, for instance, as much as manifolds \textit{\cite{Sternin}}. Our present study is intended to open a gate for an inflow of methods of differential equations, discrete equations and any other methods of dynamics research to the realm of fractals. This will help not only to investigate fractals but also to make their diversity richer and ready for intensive applications. Formally speaking, in our investigation we join dynamics of iterations, which can be called inner dynamics, with outer dynamics of differential and discrete equations. We are strongly confident that what we have suggested can be in the basement of fractals through dynamical systems, differential equations, optimization theory, chaos, control, and other fields of mathematics related to the dynamics investigation. The concept of fractals has already many applications, however, the range would be significantly enlarged if all the power of differential equations will be utilized for the structures. This is why our suggestions are crucial for fractals considered in biology, physics, city planning, economy, image recognition, artificial neural networks, brain activity study, chemistry, and all engineering disciplines \textit{\cite{Batty,Kaandorp,Zhao,Takayasu,Pietronero,Vehel}}, i.e. in every place, where the geometrical objects in physical and/or abstract sense may appear \textit{\cite{Fatou}}.

Significant results on the relation of fractals with differential equations were obtained by Kigami \textit{\cite{Kigami}} and Strichartz \textit{\cite{Strichartz}}. In the books \textit{\cite{Kigami,Strichartz}}, fractals were considered as domains of partial differential equations, but not as building blocks of trajectories. In this sense one can take advantage of the results of \textit{\cite{Kigami,Strichartz}} in the next development of our proposals. The same is true for the studies concerned with deep analyses of fractals growth performed in \textit{\cite{Vicsek,Barabasi}}.

The most important fractals for our present study are Julia sets which were discovered in 1917-1918 by Julia and Fatou, both of whom independently studied the iteration of rational functions in the complex plane \textit{\cite{Mandelbrot2}}. They established the fundamental results on the dynamics of complex functions published in their papers \textit{\cite{Julia,Fatou}}. In 1979, Mandelbrot visualized Julia sets. Although they are constructed depending on the dynamics of simple complex polynomials, the Mandelbrot set has a complicated boundary, and the Julia sets associated with points close to the boundary usually have amazing and beautiful structures.

Julia and Mandelbrot fractals are very great achievements for set theory, topology, functions theory, chaos, and real world problems. Therefore, studies in the area has to be at the frontiers of modern sciences and applications. Additionally, the development of researches on the basis of fractals is of significant importance to any possible direction starting from the basis of set theory.

In general, a fractal is defined as a set that displays self similarity and repeats the same patterns at every scale \textit{\cite{Franceschetti}}. Mandelbrot \textit{\cite{Mandelbrot1}} defined  fractals as sets for which the Hausdorff dimension is either fractional or strictly exceeds the topological dimension. Based on this, self-similarity can occur in fractals, but, in general, a fractal need not be self-similar. The easiest way to indicate  that a set has fractional dimension is through self-similarity \textit{\cite{Crownover}}.

Besides the iteration of rational maps, there are various ways to construct fractal shapes. The well known self-similar fractals like Sierpinski gasket and Koch curve are constructed by means of a simple recursive process which consists in iteratively removing shrinking symmetrical parts from an initial shape. These types of geometrical fractals can also be produced by an iterated function system \textit{\cite{Hutchinson,Barnsley}}, which is defined as collections of affine transformations.

In our paper, we make a possible study of mapping fractals, which is simple from one side since it relates to classical functions, but from another side it is a developed one since we apply the mapping function in a new manner, which nevertheless still relates to the original ideas of Julia. 

There are two sides of the fractal research related to the present paper. The first one is the Fatou-Julia Iteration (FJI) and the second one is the proposal by Mandelbrot to consider dimension as a criterion for fractals. In our study, both factors are crucial as we apply the FJI for the construction of the sets and the dimension factor to confirm that the built sets are fractals. In previous studies the iteration and the dimension factors were somehow separated, since self-similarity provided by the iteration has been self-sufficient to recognize fractals, but in our research the similarity is not true in general. We have to emphasize that there is a third player on the scene, the  modern state of computers' power. Their roles are important for the realization of our idea exceptionally for continuous dynamics. One can say  that the instrument is at least of the same importance for application of our idea to fractals as for realization of Fatou-Julia iteration in Mandelbrot and Julia sets. Nevertheless, we expect that the present study will significantly increase the usage of computers for fractal analysis. Moreover, beside differential equations, our suggestions will effect the software development for fractals investigation and applications \textit{\cite{Encarnacao, Peitgen}}.

Studying the problem, we have found that fractal-like appearances can be observed in ancient natural philosophy. Let us consider the Achilles and tortoise in the Zeno's Paradox \textit{\cite{Goswami}}, (see Fig. \ref{TorAchill}).

\end{multicols}
\begin{figure}[H]
\centering	
{\setlength{\fboxsep}{0pt}%
\setlength{\fboxrule}{1.5pt}%
\fcolorbox{red}{white}{
	\begin{minipage}{0.48\textwidth}
		\vspace{0.2cm} \hspace{-1.3cm}
		\centering
		\subfigure[Achilles and the tortoise dynamics]{\includegraphics[width = 2.20in]{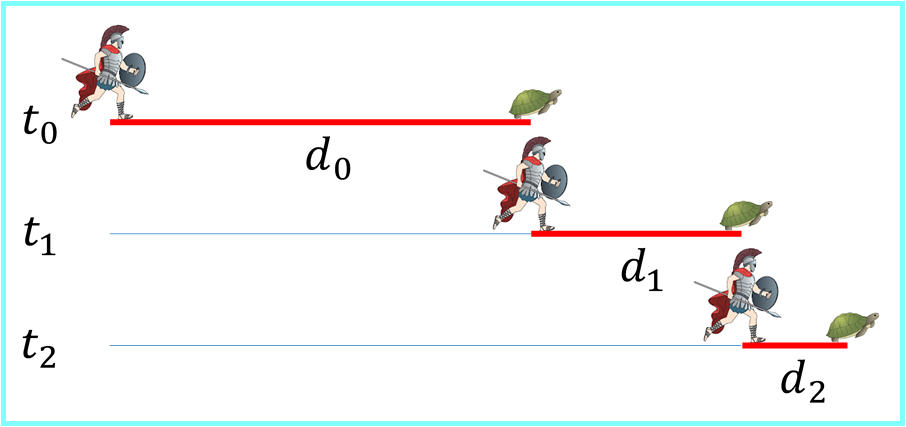}\label{TorAchill}} \\
		\vspace{-0.0cm} \hspace{-1.5cm}	
		\subfigure[The state, $S_0$, of the dynamics ]{\includegraphics[width = 2.85in]{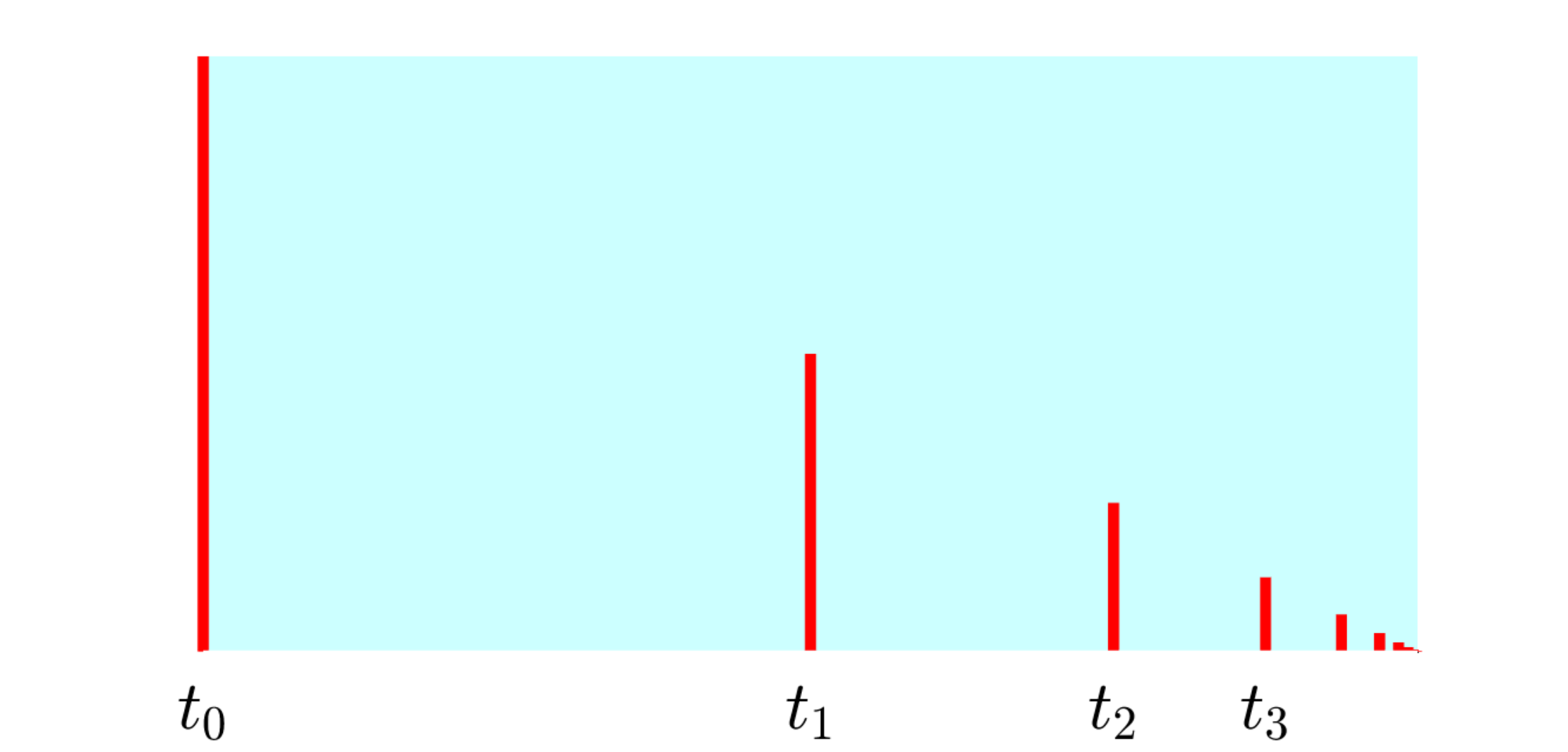}\label{TorAchDiag}}
	\end{minipage}
	\hspace{-1.2cm}
	\begin{minipage}{0.50\textwidth}
		\centering 				
		\subfigure[Pseudo-fractal trajectory]{\includegraphics[width = 3.5in]{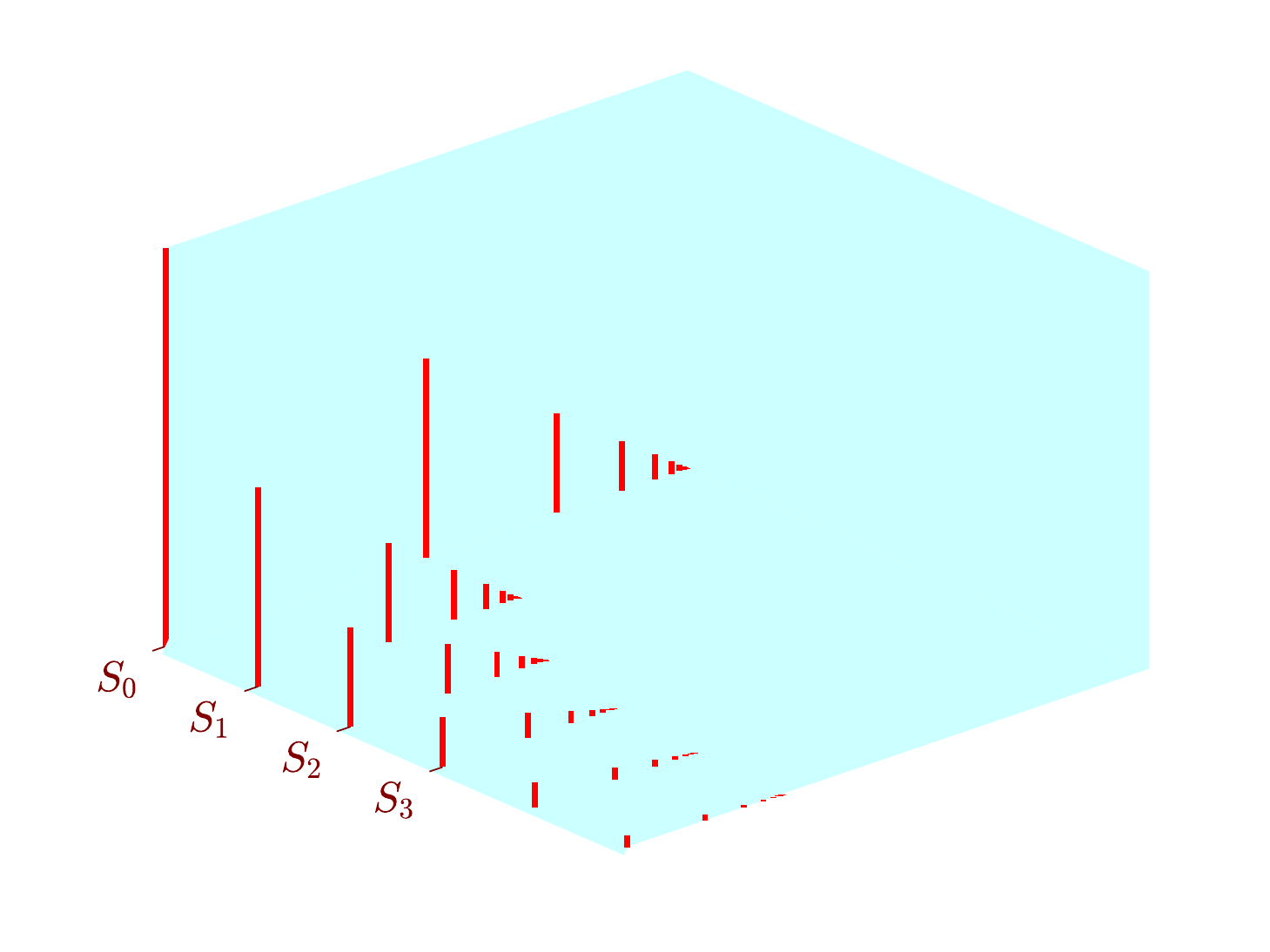}\label{UnionFrac}}
	\end{minipage}}	}
\caption{The dynamics of the Zeno's Paradox}
\label{TorAchillParadox}
\end{figure}
\begin{multicols}{2}

In the paradox, Achilles is observed at the initial moment $ t_0=0 $ with the distance $ d_0 $ from the tortoise. Suppose that Achilles runs at a constant speed, two times faster than the tortoise, then he would reach the previous position of the tortoise at moments $ t_1, \, t_2=3 t_1/2, \, t_3=7 t_1/4, ... $ with distances $ d_1=d_0/2, \, d_2=d_0/4, \, d_3=d_0/8, ... $ from the tortoise, respectively. Now contemplate Fig. \ref{TorAchDiag}, where the heights of the red lines are proportional to the distance of Achilles from the tortoise at the fixed moments, and denote the diagram by $S_0$. The set $S_0$ demonstrates the entire dynamics for $t \ge t_0.$ Fix $i \ge 0,$ and let $S_i$ be the similar diagram which consists of all the lines for the moments which are not smaller than $t_{i}$.
	
Let us consider the collection of the states $ \{S_i\}, i \ge 0 $. One can assume that there exists a map $\mathcal{B}$ such that the equations
\[ S_{i+1}= \mathcal{B} \, S_i, \; i = 0,1,2,\ldots, \]
which symbolize a dynamics, are valid. It is easily seen that $S_0$ is self-similar to each of its parts $S_i$, $i>0$. Nevertheless,  the Hausdorff dimension of the set $ S_0$ is equal to one. For this reason, we call  $S_i$, $i \ge 0,$ \textit{pseudo-fractals}, due to the similarity. The trajectory $\{S_i\}, i\geq 0,$ is also a pseudo-fractal. The sketch of the trajectory is seen in Fig. \ref{UnionFrac}.

Our present study is an extension of the ancient paradigm, since we will investigate dynamics having all points of a trajectory as well as the trajectory itself fractals.
	
\section*{\normalsize \color{red}  FATOU-JULIA ITERATIONS}

The core of the present study is built through FJI and the quadratic map $ P: \mathbb{C} \times \mathbb{C} \rightarrow \mathbb{C} $ defined by 	
\begin{equation} \label{MainQuadPoly}
P(z,c)=z^2+c,
\end{equation}
where $ \mathbb{C} $ denotes the set of complex numbers.

Consider the equation  $ z_{n+1}=P(z_n,c), \; n \geq 0 $. The points $ z_0 \in \mathbb{C} $ are included in a fractal $ \mathcal{F} $ depending on the boundedness of the sequence $z_n,$ and we say that the fractal $ \mathcal{F} $ is constructed by FJI. The so-called filled-in Julia set, $ \mathcal{K}_c $, is constructed by including only the points $ z_0 \in \mathbb{C} $ such that the sequence $z_n$ is bounded \textit{\cite{Ahlfors}}. Moreover, in the simulation, those points $ z_0 \in \mathbb{C} $ where $ \{z_n\} $ is divergent are colored in a different way, correspondingly to the rate of divergence. The term Julia set $ \mathcal{J}_c $,  usually denotes the boundary of the filled Julia set, i.e., $ \mathcal{J}_c= \partial \mathcal{K}_c $.

The Mandelbrot set, $ \mathcal{M} $, is also generated by a Fato-Julia iteration. In this case, we consider the equation $ z_{n+1}=P(z_n,c) $,  and include in the set $ \mathcal{M} $ the points $ c \in \mathbb{C} $ such that  $ \{z_n(c)\}, \, z_0(c)=0, $ is bounded.  Here again, the points $ c \in \mathbb{C} $ corresponding to divergent sequences $z_n$ are plotted in  various colors depending  on the rate of the divergence.  

\section*{\normalsize \color{red}HOW TO MAP FRACTALS}
	
To describe our way for mapping of fractals, let us consider a fractal set $ \mathcal{F} \subseteq A \subset \mathbb C, $ constructed by the following FJI, 
\begin{equation} \label{F-J It}
	z_{n+1}=F(z_n),
\end{equation}
where  $ F: A  \rightarrow A $ is  not  necessarily  a rational map.  We suggest  that the original  fractal  $ \mathcal{F} $  can be transformed ``recursively'' into a new fractal set. For that purpose, we modify the  FJI, and consider iterations to be of the form
\begin{equation} \label{Map It}
	f^{-1}(z_{n+1})=F\big(f^{-1}(z_n)\big),
\end{equation}
or explicitly,
\begin{equation} \label{Map Itd}
	z_{n+1}=f\Big(F\big(f^{-1}(z_n)\big)\Big), 
\end{equation}
where $f$ is a one-to-one map on $ A $. Next, we examine the convergence of the sequence $ \{ z_n \} $ for each $ z_0 \in f(A). $ Denote by $\mathcal{F}_f$ the  set which contains only the points $ z_0 $ corresponding to the bounded sequences. Moreover, other points can be plotted in different colors depending  on the rate of the divergence of $ \{ z_n \} $. To distinct the iterations 	(\ref{Map Itd}) from the Fatou-Julia iterations let us call the first ones {\it Fractals Mapping Iterations} (FMI). It is clear that FJI is a particular FMI, when the function is the identity map. The mapping of fractals is a difficult problem which depends on infinitely long iteration processes, and has to be accompanied with sufficient conditions to  ensure  that the  image is again fractal. 

The next theorem is the main instrument for the detection of fractal mappings. Accordingly, we call it {\it Fractal  Mapping Theorem} (FMT).
 
\begin{theorem} \label{Thm2}
	If $ f $ is a bi-Lipschitz function, i.e. there exist numbers $ l_1, l_2 > 0 $ such that
	\begin{equation}  \label{Bi-Lip cond}
	l_1 |u-v| \leq |f(u)-f(v)| \leq l_2 |u-v|
	\end{equation}
	for all $ u, v \in A,$   then $ \mathcal{F}_f=f(\mathcal{F}) $.
\end{theorem}
\begin{proof}
	Fix an arbitrary $w \in \mathcal{F}_f$. There exists a bounded sequence $ \{w_k\} $ such that	$ w_0=w $ and $ f^{-1}(w_{k+1})=F(f^{-1}(w_k) $. Let us denote $ z_k=f^{-1}(w_k) $. Our purpose is to show that $\left\{ z_k \right\}$ is a bounded sequence. Indeed
	\[ |z_k-z_0|= |f^{-1}(w_k) -f^{-1}(w_0)| \leq \frac{1}{l_1}|w_k-w_0|. \]
	Hence, the boundedness of $ \{w_k\} $ implies the same property for $ \{z_k\} $, and therefore, we have $ z_0=f^{-1}(w) \in \mathcal{F} $.
		
	Now, assume that $ w \in f(\mathcal{F}) .$ There is $ z \in \mathcal{F} $  such that $ f(z)=w $ and a bounded sequence $ \{z_k\} $ such that 	$ z_0=z $ and $ z_{k+1}=F(z_k) $. Consider, $ w_0 =w $ and $ w_{k}=f(z_k), \, k\geq0 $. It is clear that the sequence $ \{w_k\} $ satisfies the iteration (\ref{Map It}) and moreover
	\[ |w_k-w_0|= |f(z_k) -f(z_0)| \leq l_2 |z_k-z_0|. \]
	Consequently, $ \{w_k\} $ is bounded, and $ w \in \mathcal{F}_f $.
\end{proof}
	
The following two simple propositions  are required.
	
\begin{lemma} \label{Lem1} \cite{Falconer}
	If $f$ is a bi-Lipschitz function, then 
	\[ \dim_{H} f(A)  = \dim_{H} A, \] where $\dim_H$ denotes the Hausdorff dimension.
\end{lemma}
\begin{lemma} \label{Lem2}
	If $ f: A \rightarrow \mathbb C$ is a homeomorphism, then it maps the boundary of $ A $ onto the boundary of $ f(A).$
\end{lemma}
	
It is clear that a bi-Lipschitz function is a homeomorphism.
	
Shishikura \textit{\cite{Shishikura}} proved that the Hausdorff dimension of the boundary of the Mandelbrot set is $2$. Moreover, he showed that the Hausdorff dimension of the Julia set corresponding to $ c \in \partial \mathcal{M} $ is also $2$.

It implies from the above discussions that if $f$ is a bi-Lipschitz function and $\mathcal{F}$ is either a Julia set or the boundary of the Mandelbrot set, then their images $\mathcal{F}_f $ are fractals. In what follows, we will mainly use functions, which are bi-Lipschitzian except possibly in neighborhoods of single points. 

Now, we apply FMI to a Julia set $\mathcal{J},$ and the iteration will be in the  form
\begin{equation} \label{MapJul}
	f^{-1}(z_{n+1})=\big[f^{-1}(z_n)\big]^2+c,	
\end{equation}
with various functions $f$ and values of $c.$  The resulting    fractals $\mathcal{J}_f = f( \mathcal{J})$ are depicted in Figs. \ref{FMI-MJ1} and \ref{FMI-MJ2}. They are mapped by $ f(z)=\cos^{-1}\big(\frac{1}{z}-1\big) $, $ c=-0.7589+0.0735i $, and  $f(z)=\big(\sin^{-1} z\big)^\frac{1}{5}$, $c=-0.175-0.655i$ from the Julia sets in Figs. \ref{FMI-J1} and \ref{FMI-J2}, respectively.

For mapping of the Manderbrot set, we propose the FMI 
\begin{equation} \label{MapMandel}
	z_{n+1}=z_n^2+f^{-1}(c).	
\end{equation}
Along the lines of the proof of Theorem \ref{Thm2}, one can show that if the map $f$ is bi-Lipschitzian, then the iteration (\ref{MapMandel}) defines the relation $ f(\mathcal{M})=\mathcal{M}_f $, where $ \mathcal{M}_f$ is a new fractal. Figure \ref{FMI-MM} shows  a fractal mapped by $ f(c)=\big(\frac{1}{c}-1\big)^\frac{1}{2} $ from the Mandelbrot set in Fig. \ref{FMI-M}.

\begin{figure}[H]
	\centering	
	{\setlength{\fboxsep}{0pt}%
	\setlength{\fboxrule}{1.5pt}%
	\fcolorbox{red}{white}{
		\begin{minipage}{0.465\textwidth}
			\subfigure[]{\hspace{-0.3cm}\includegraphics[width = 1.7in]{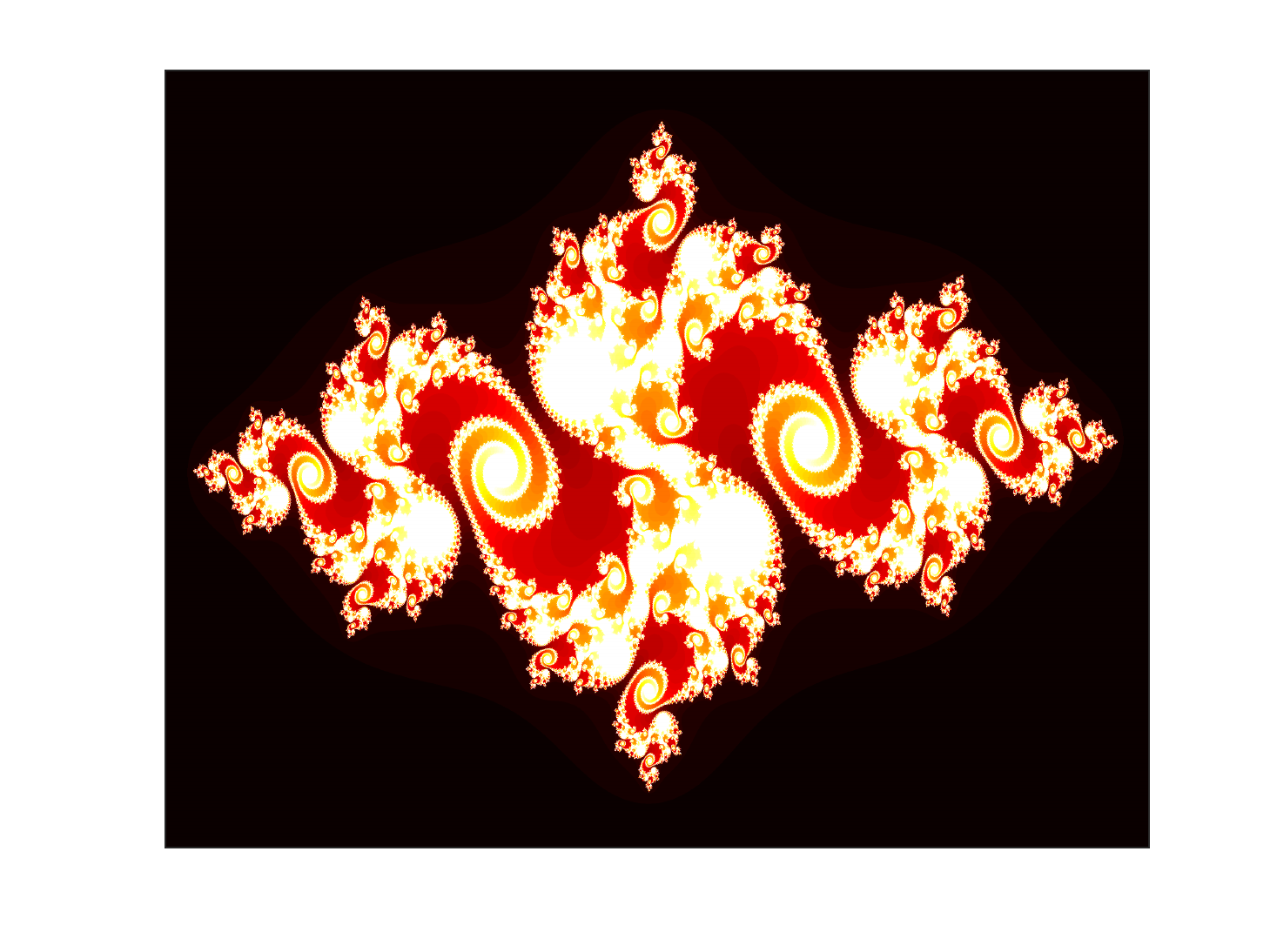}\label{FMI-J1}}
			\subfigure[]{\hspace{-0.3cm}\includegraphics[width = 1.7in]{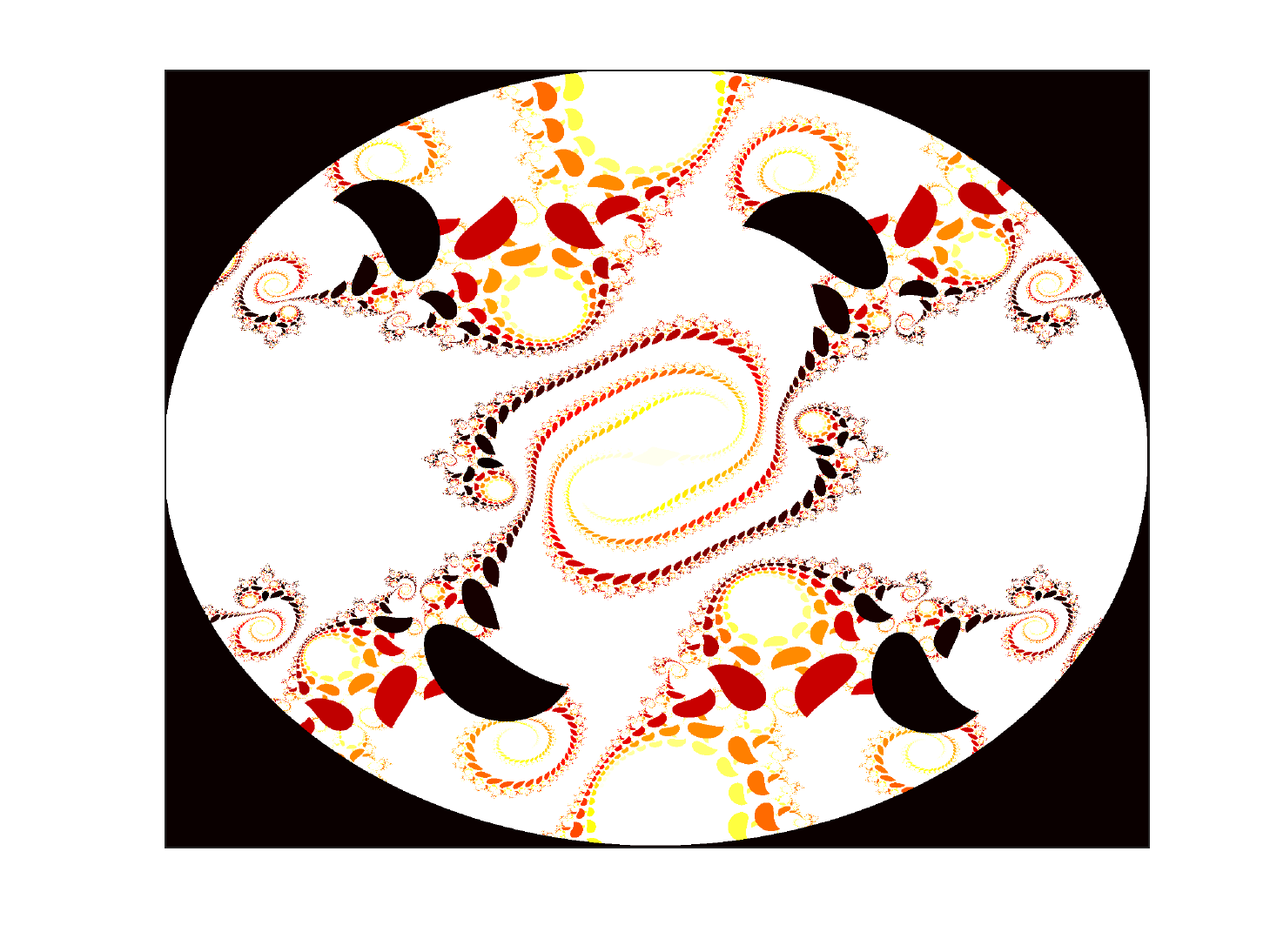}\label{FMI-MJ1}} \\
			\subfigure[]{\hspace{-0.3cm}\includegraphics[width = 1.7in]{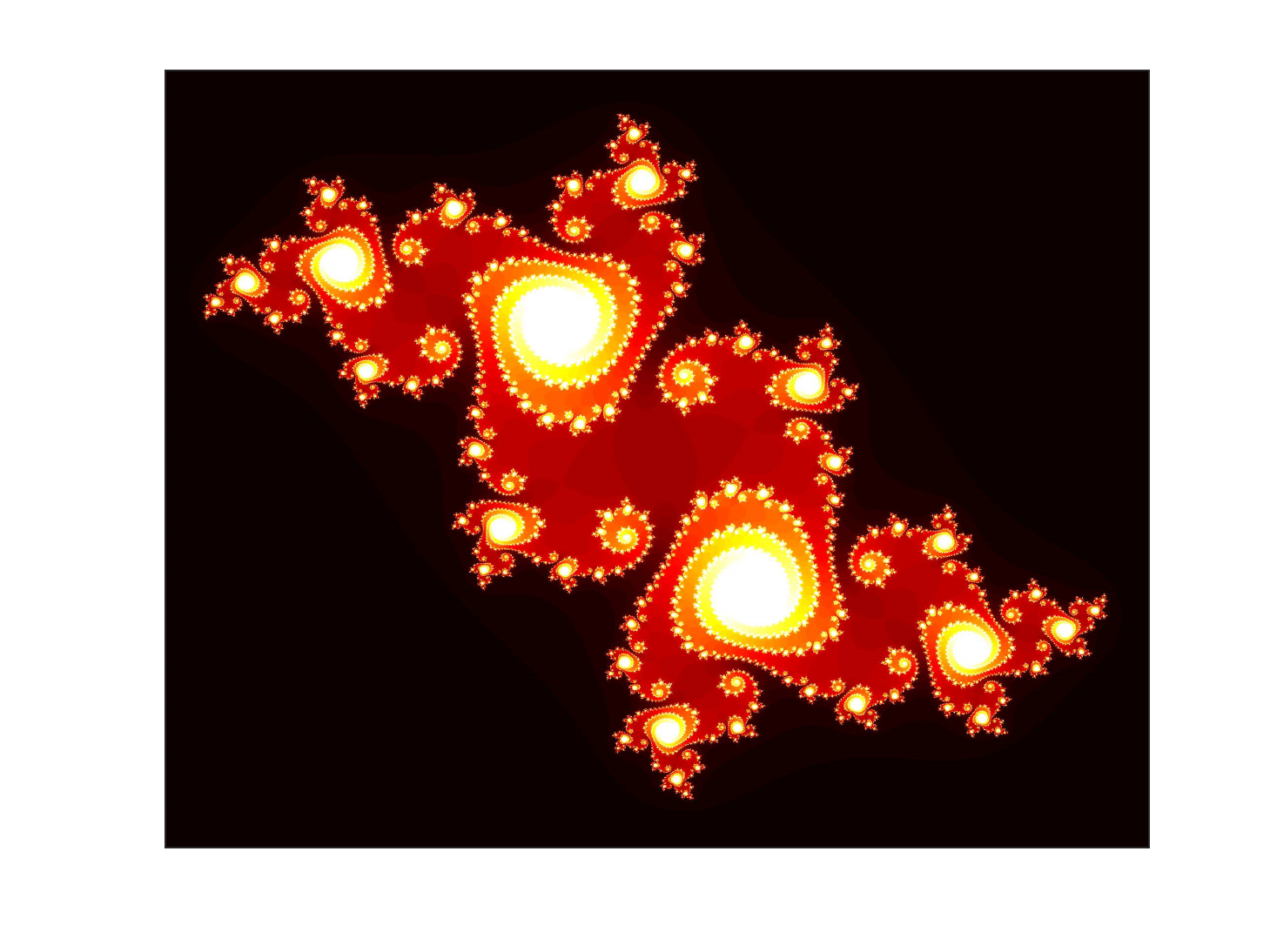}\label{FMI-J2}}
			\subfigure[]{\hspace{-0.3cm}\includegraphics[width = 1.7in]{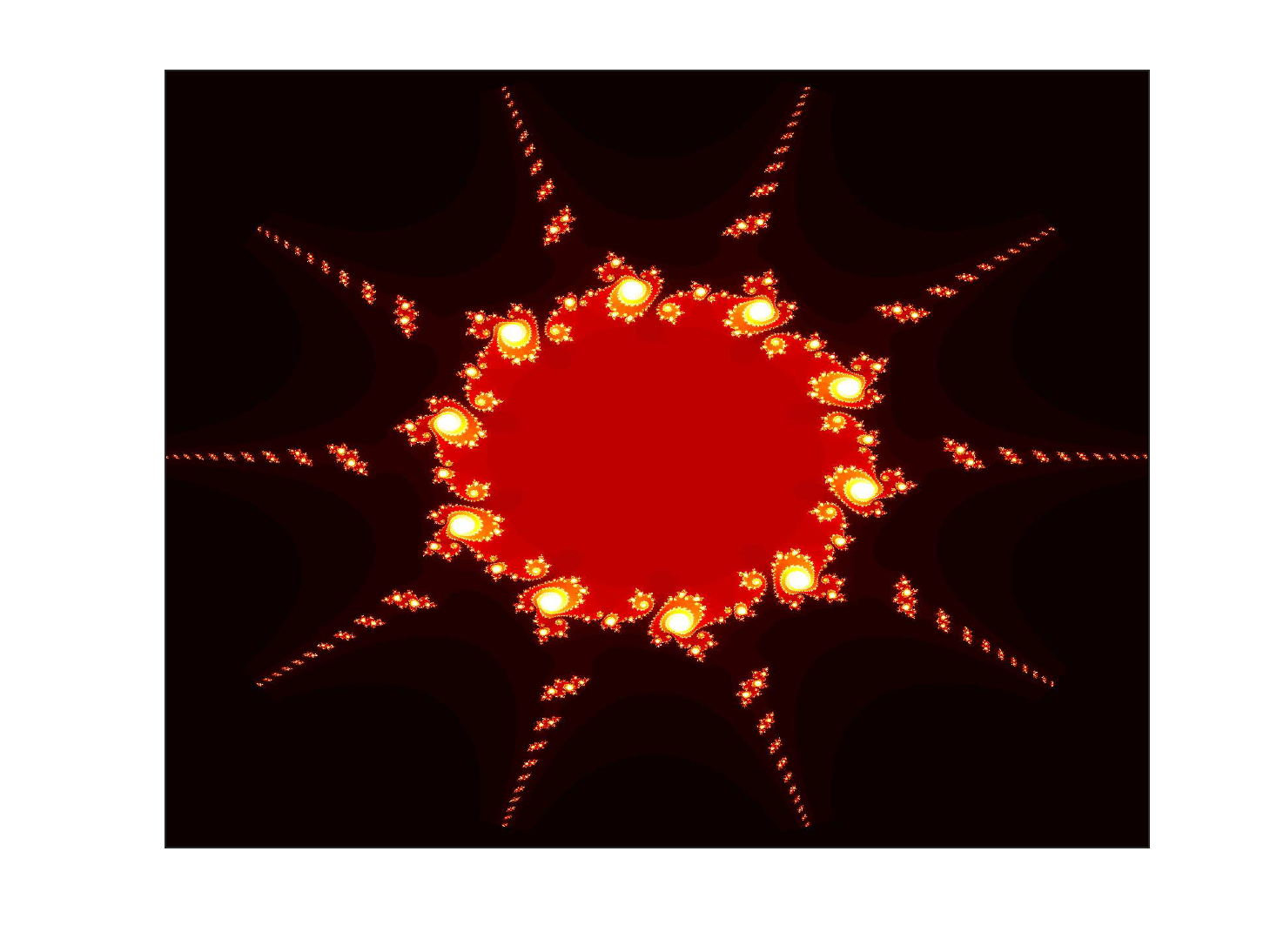}\label{FMI-MJ2}}\\
			\subfigure[]{\hspace{-0.3cm}\includegraphics[width = 1.7in]{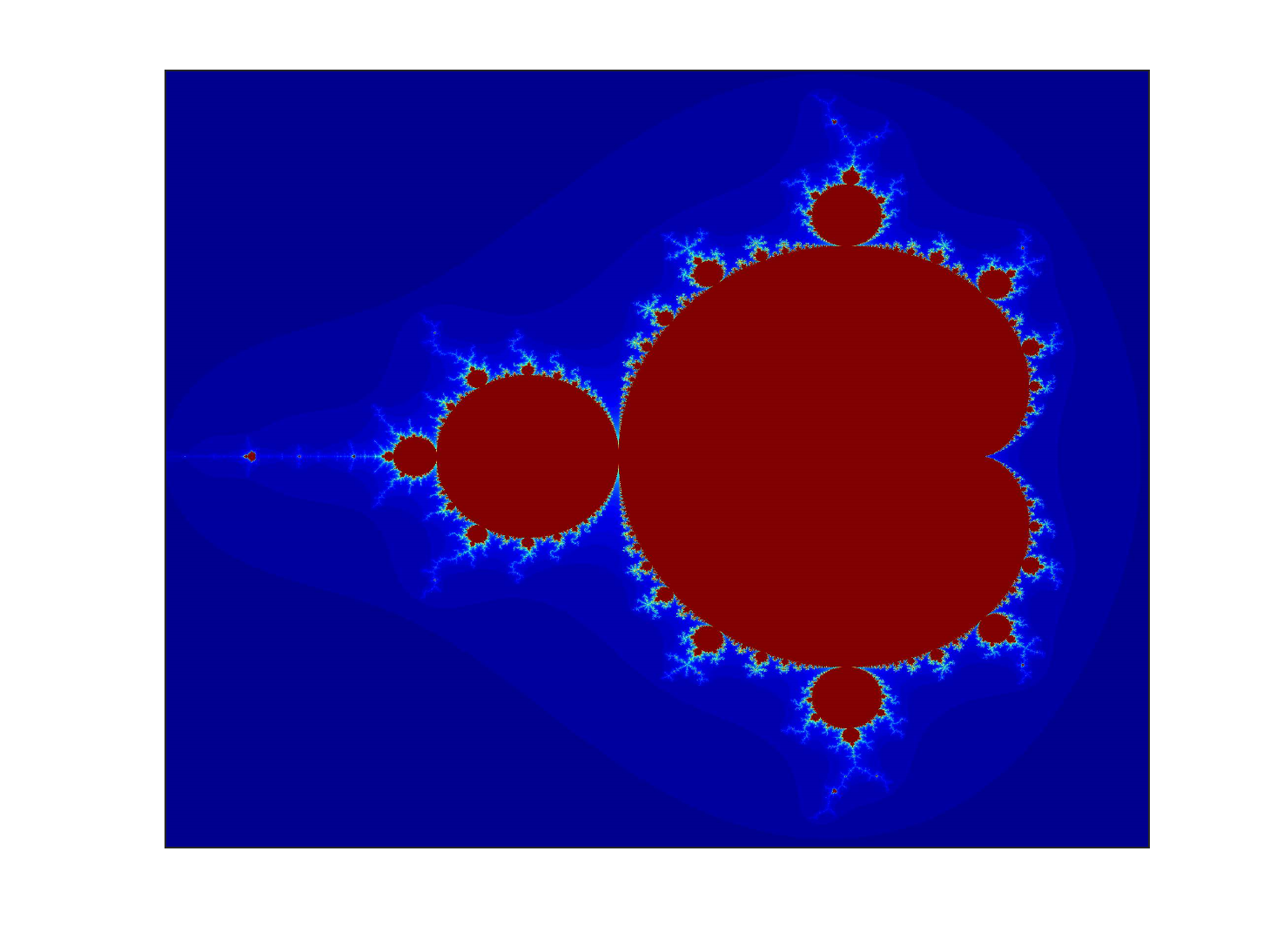}\label{FMI-M}}
			\subfigure[]{\hspace{-0.3cm}\includegraphics[width = 1.7in]{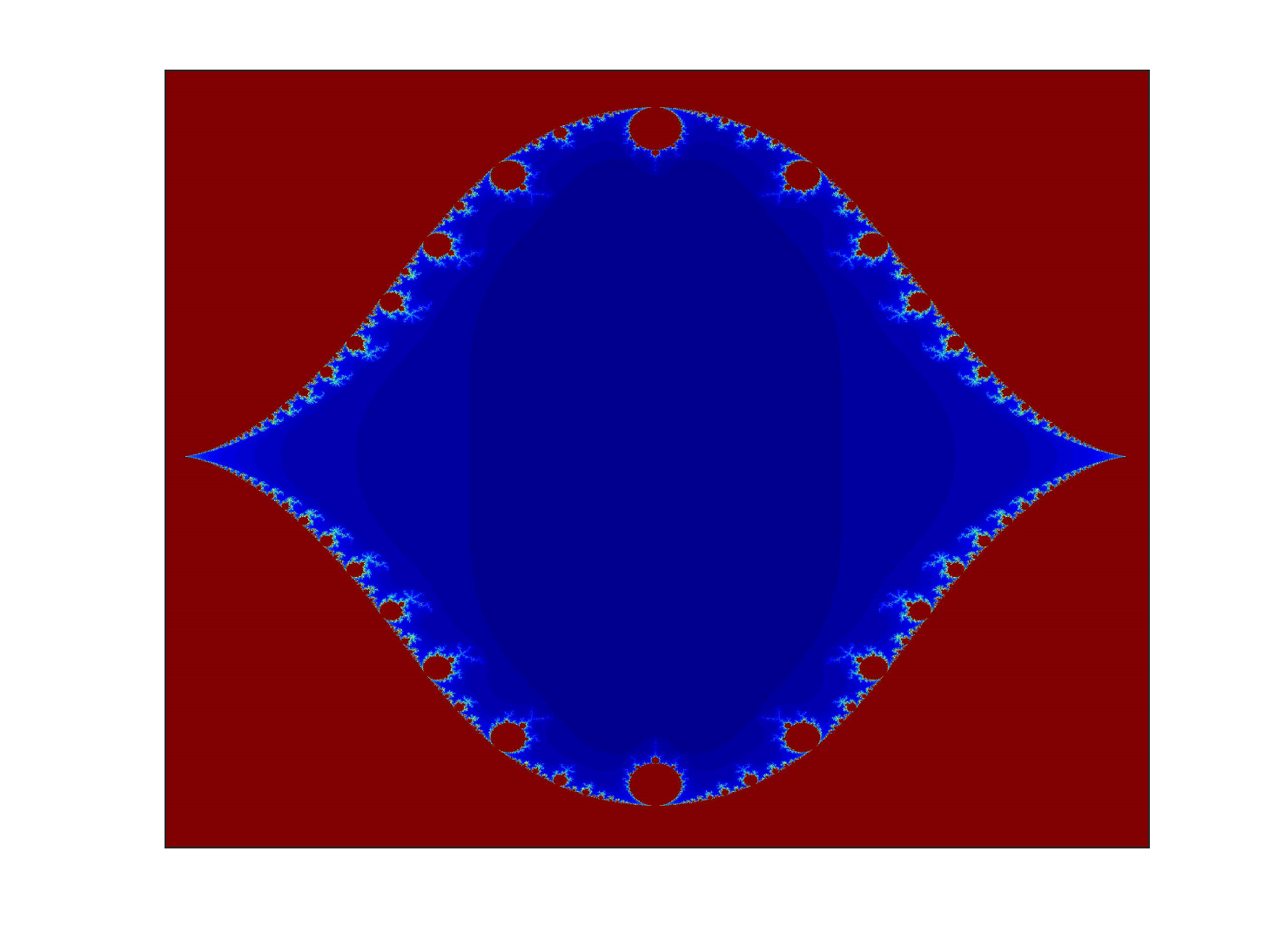}\label{FMI-MM}}
	\end{minipage}}}
	\caption{Julia and Mandelbrot sets with  their images}
	\label{FMI}   				
\end{figure}

\section*{\normalsize \color{red}  DISCRETE DYNAMICS}
	
Discrete fractal dynamics means simply iterations of mappings  introduced in the last section.
Let us consider a discerete dynamics with a bi-Lipschitz iteration function $f$ and a Julia set  $ \mathcal{J}_{0} =   \mathcal{J}$  as an initial fractal for the dynamics. The trajectory
\[ \mathcal{J}_0 ,\mathcal{J}_1,\mathcal{J}_2,\mathcal{J}_3,\ldots, \]
is obtained by the FMI  
\begin{equation} \label{DescJul2}
	z_{n+1}=f^{k}\Big([f^{-k}(z_n)]^2+c\Big)	
\end{equation}
 such that $\mathcal{J}_{k+1} = f( \mathcal{J}_k)$, $k = 0,1,2,3, \ldots$. The last equation is a fractal propagation algorithm.
	 	 
 Figure \ref{DiscDyn} shows the trajectory and its points at $ k=1 $ and $ k=5 $ for the function $ f(z)=z^2+a c+b $ with $ a=0.6 $, $ b=0.02-0.02i $ and $ c=-0.175-0.655i $.
\begin{figure}[H]
	\centering
	{\setlength{\fboxsep}{0pt}%
	\setlength{\fboxrule}{1.5pt}%
	\fcolorbox{red}{white}{
	\begin{minipage}{0.46\textwidth}
	\subfigure[Discrete fractal trajectory]{\includegraphics[width = 3.2in]{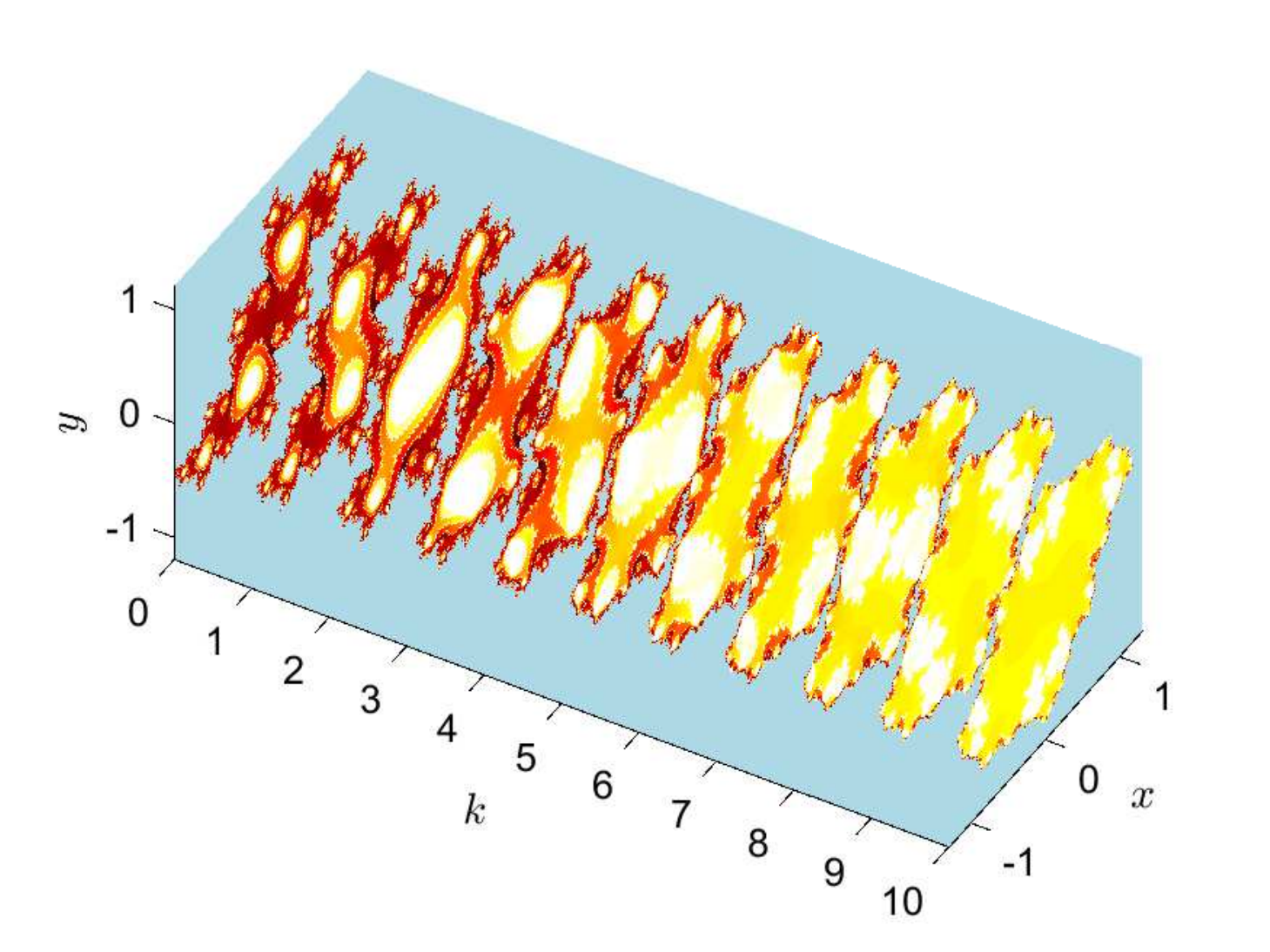}\label{DescDyT}}\\
	\subfigure[$ k=1 $]{\includegraphics[width = 1.6in]{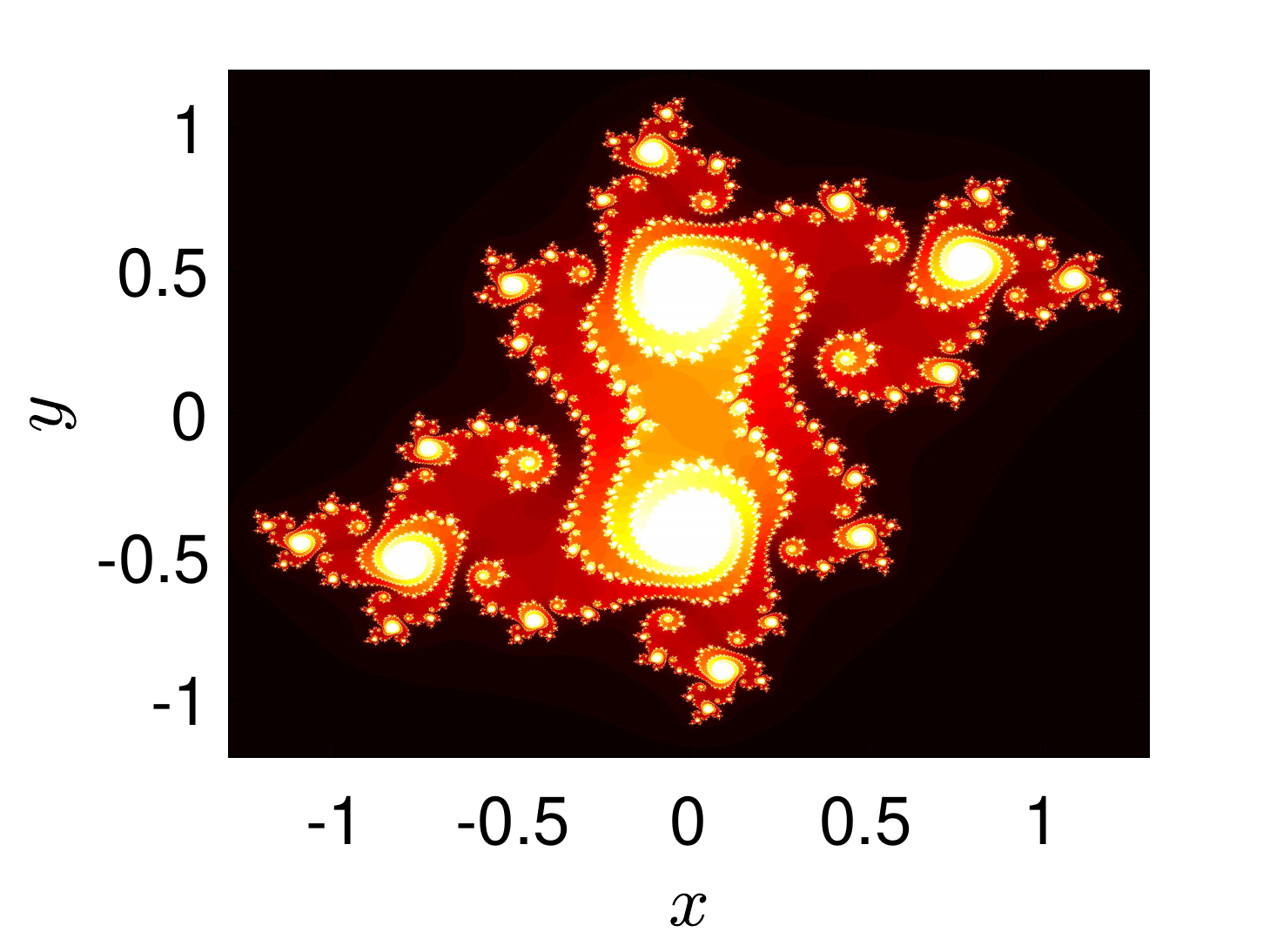}\label{DescDySa}}
	\hspace{-0.2cm}
	\subfigure[$ k=5 $]{\includegraphics[width = 1.6in]{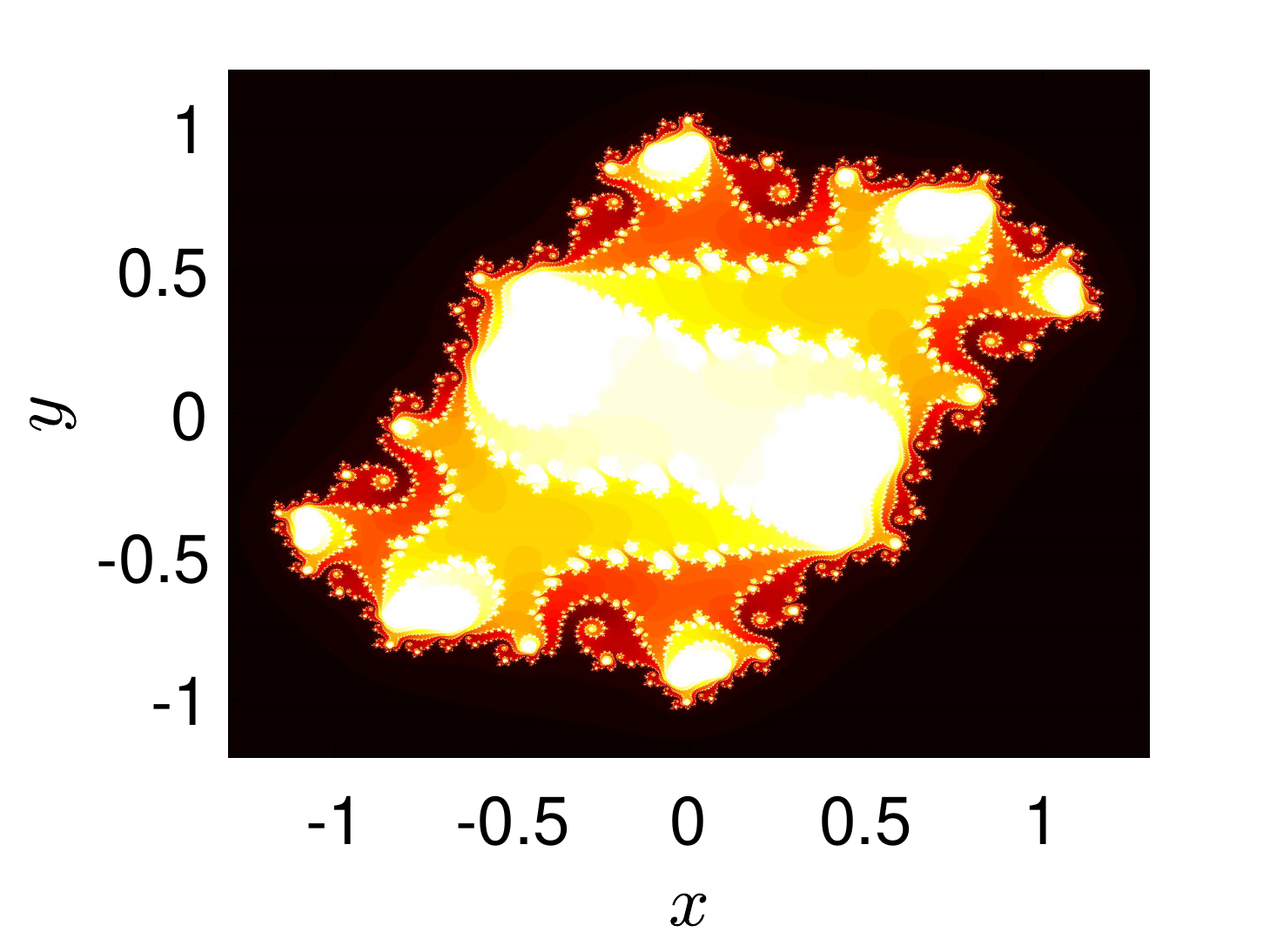}\label{DescDySb}}
	\end{minipage}}}
	\caption{The discrete trajectory}
	\label{DiscDyn}	
\end{figure}
	
\section*{\normalsize\color{red}  CONTINUOUS DYNAMICS}
	
To demonstrate a continuous dynamics $A_t$ with real parameter $t$ and fractals, we use the differential equation	
\begin{equation} \label{Diff}
	\frac{dz}{dt}=g(z), 	
\end{equation}	
such that $ A_t z = \phi(t,z),$ where $\phi(t,z)$ denotes the solution of (\ref{Diff}) with $\phi(0,z) = z.$ Thus, we will construct dynamics of sets  $A_t \mathcal{F},$ where a fractal $\mathcal{F}$ is the initial value. To be in the course of the previous sections, we define a map $f(z) =  A_t z$ and the equation  
\begin{equation*} \label{ContJul1}
	A_{-t}(z_{n+1})=[A_{-t}(z_n)]^2+c.	
\end{equation*}
Thus the  FMI (\ref{Map Itd}) in this case will be in the form
\begin{equation} \label{ContJul2}
	z_{n+1}=A_{t}\Big([A_{-t}(z_n)]^2+c\Big).	
\end{equation}	
In what follows, we assume that the map $ A_t $ is bi-Lipschitzian. This is true, for instance, if the function $ g $ in (\ref{Diff}) is Lipschitzian. Then the  set $ A_t \mathcal{F} $ for each fixed $ t $ is a fractal determined by the FMI,  and we can say about continuous fractal dynamics.
	 
As an example we consider the differential equation $ dz/dt=-z, \; 0 \leq t \leq 1 $, with the flow $ A_t z=z e^{-t} $. It represents a contraction mapping  when it is applied to the iteration (\ref{ContJul2}), whereas the unstable dynamical system $ A_t z=z e^{t} $ corresponding to the differential equation $ dz/dt=z $ represents an expansion mapping.
	
\end{multicols}	
\begin{figure}[H]
	{\setlength{\fboxsep}{0pt}%
	\setlength{\fboxrule}{1.5pt}%
	\fcolorbox{red}{white}{
	\begin{minipage}{0.32\textwidth}
		\hspace{-1.0cm}
		\subfigure[$ \mathcal{J} e^{-t} $]{\includegraphics[width = 3.1in]{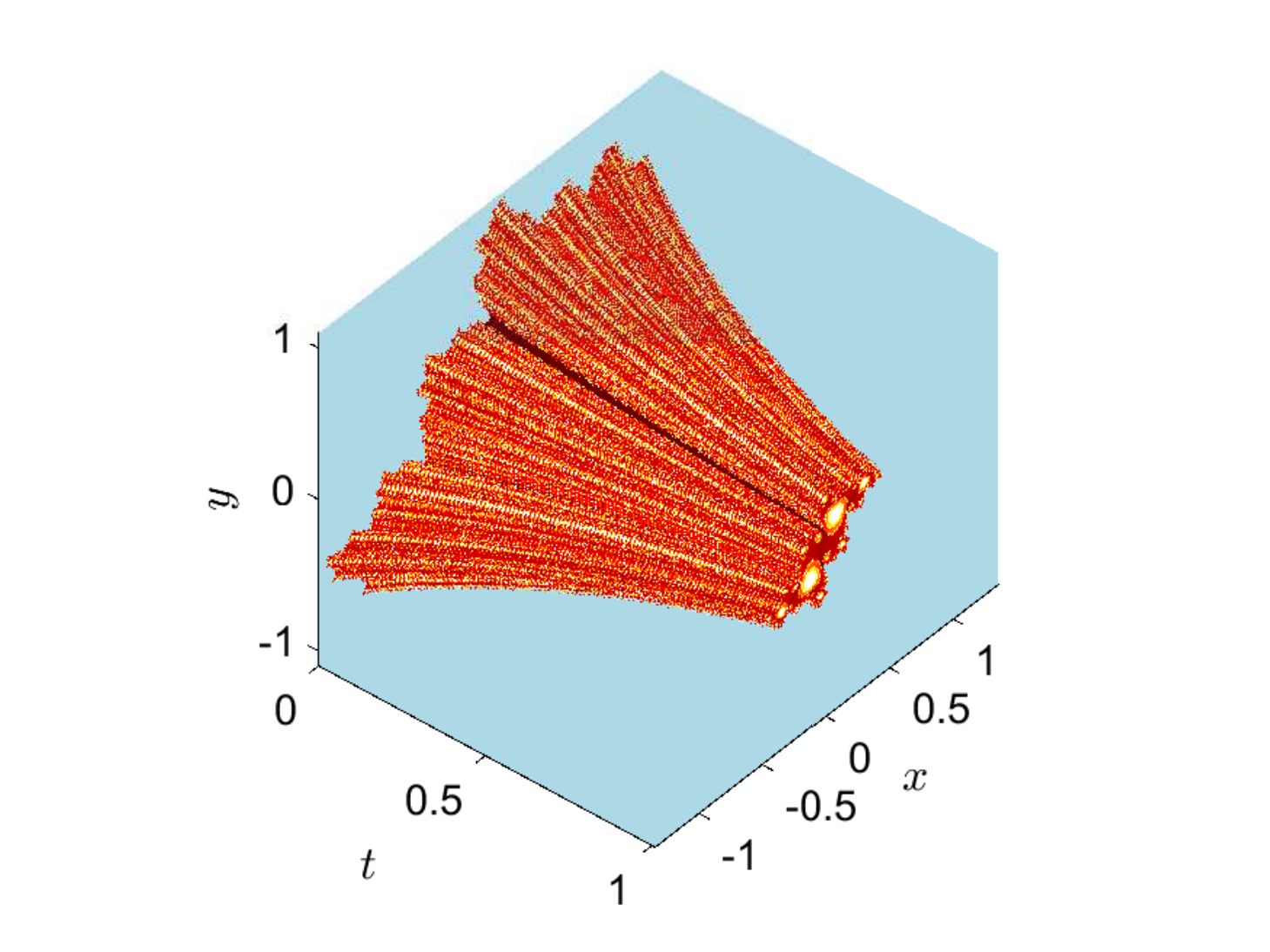}\label{ContDynSimpleCotra}}
	\end{minipage}
	\begin{minipage}{0.32\textwidth}
		\hspace{-0.6cm}
		\subfigure[$ \mathcal{J} e^{t} $]{\includegraphics[width = 3.1in]{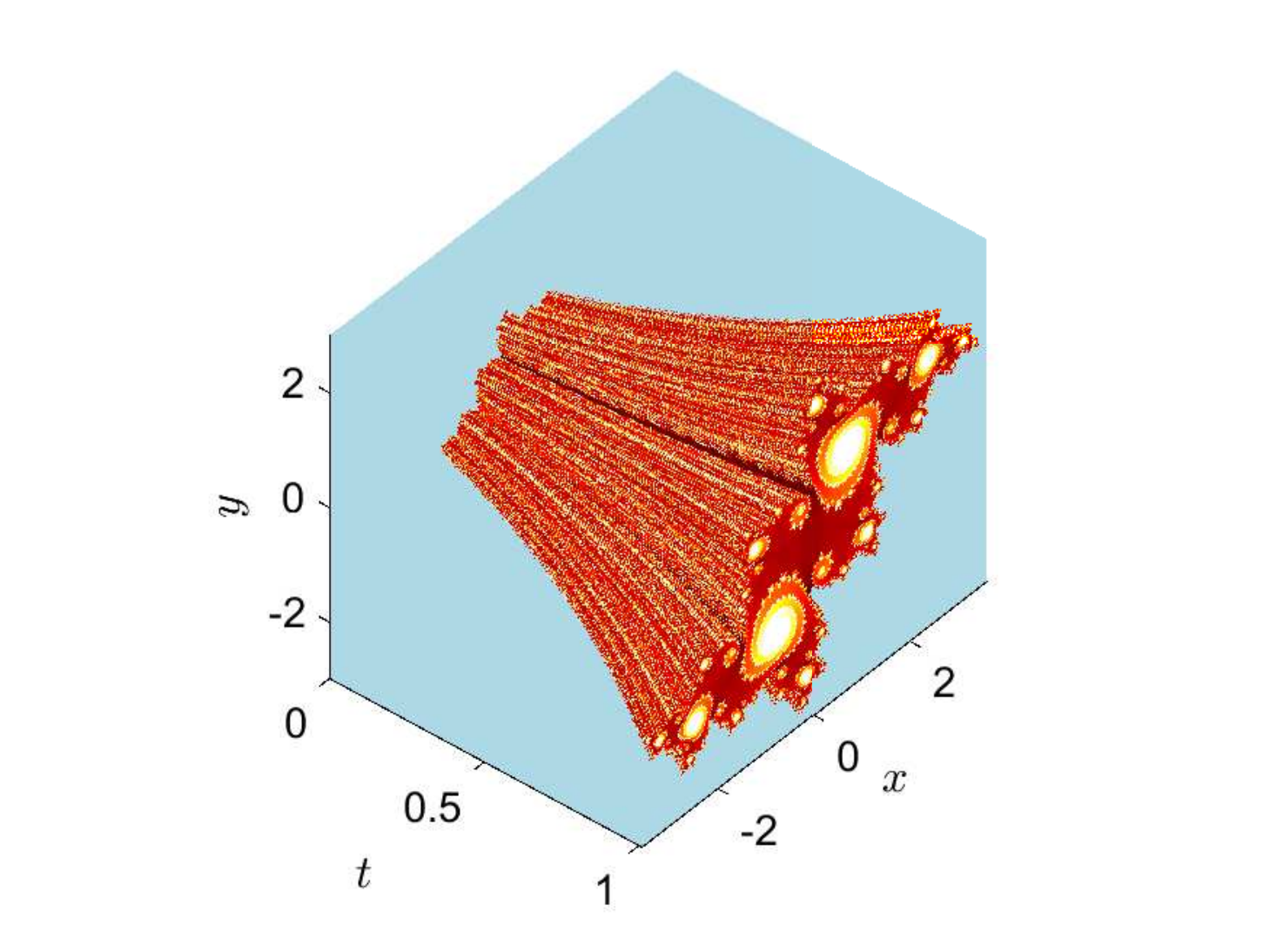}\label{ContDynSimpleExpan}}
	\end{minipage}
	\begin{minipage}{0.32\textwidth}
		\centering
		\subfigure[$ t=0 $]{\includegraphics[width = 1.6in]{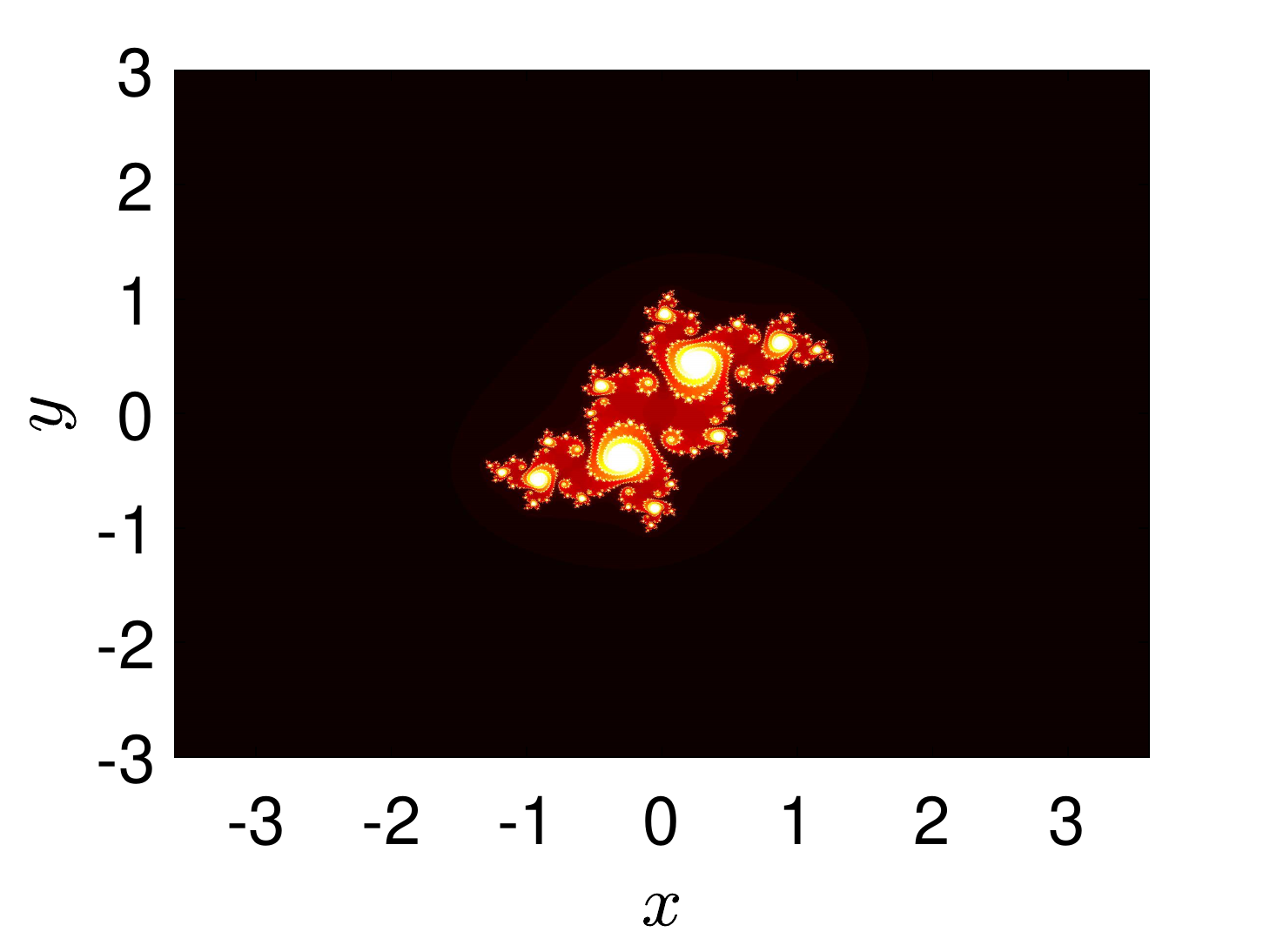}\label{ContDynSimpleExpanS1}} \\ 
		\subfigure[$ t=1 $]{\includegraphics[width = 1.6in]{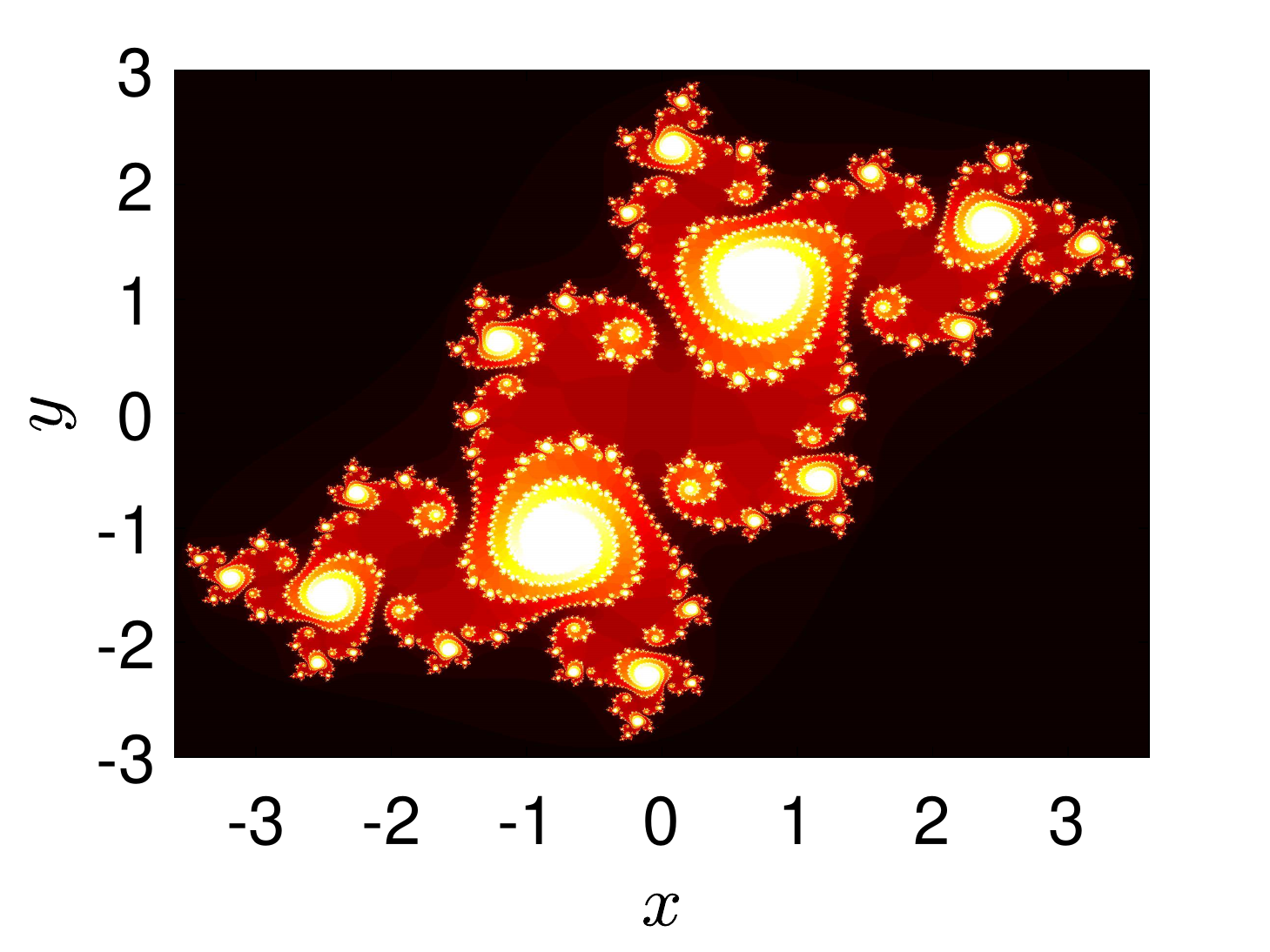}\label{ContDynSimpleExpanS2}}
	\end{minipage}}}
\caption{Fractals of the continuous dynamics}
\label{ContDynSimple}   				
\end{figure}
\begin{multicols}{2}	
	
Figure \ref{ContDynSimple} (a) and (b) contain fractal trajectories of the dynamics with the initial Julia set $\mathcal{J},$ corresponding to $ c=-0.175-0.655i .$ The initial fractal and the point of the expansion at $t=1$ are seen in parts (c) and (d) of the figure, respectively.
	
Now, we will focus on the autonomous system of differential equations 
\begin{equation}
\begin{split}
		&\frac{dx}{dt}=-y+x(4-x^2-y^2), \\
		&\frac{dy}{dt}=x+y(4-x^2-y^2).
\end{split}	
\label{AutoODES}
\end{equation}
	
The solution of the last system in polar coordinates with initial conditions $\rho(0)=\rho_0$ and $\varphi(0)=\varphi_0$ is given by
\[\begin{split}
	&\rho(t)=2e^{4t} \Big( \frac{4}{\rho_0^2}+e^{8t}-1\Big)^{-\frac{1}{2}}, \\
	&\varphi(t)=t+\varphi_0.
\end{split}\]
Thus, the map can be constructed by
\begin{equation} \label{Dyn3}
	A_t z=x(t)+iy(t),
\end{equation}
where
\[\begin{split}
&x(t)=\rho(t) \cos(\varphi(t)), \\
&y(t)=\rho(t) \sin(\varphi(t)), \\
&z= \rho_0 \cos(\varphi_0) + i \rho_0 \sin(\varphi_0). 
\end{split}\]	
	
In Fig. \ref{ContDynAutonT}, the fractal trajectory of system (\ref{AutoODES}) is seen  with the Julia set as the initial fractal. Parts (b)-(g) of the same figure represent various points of the trajectory.
	 
\end{multicols}	
\begin{figure}[H]
	{\setlength{\fboxsep}{0pt}%
	\setlength{\fboxrule}{1.5pt}%
	\fcolorbox{red}{white}{
	\centering
	\begin{minipage}{0.30\textwidth}
			\hspace{-1.0cm}
			\subfigure[$ A_t \mathcal{J} $]{\includegraphics[width = 3.2in]{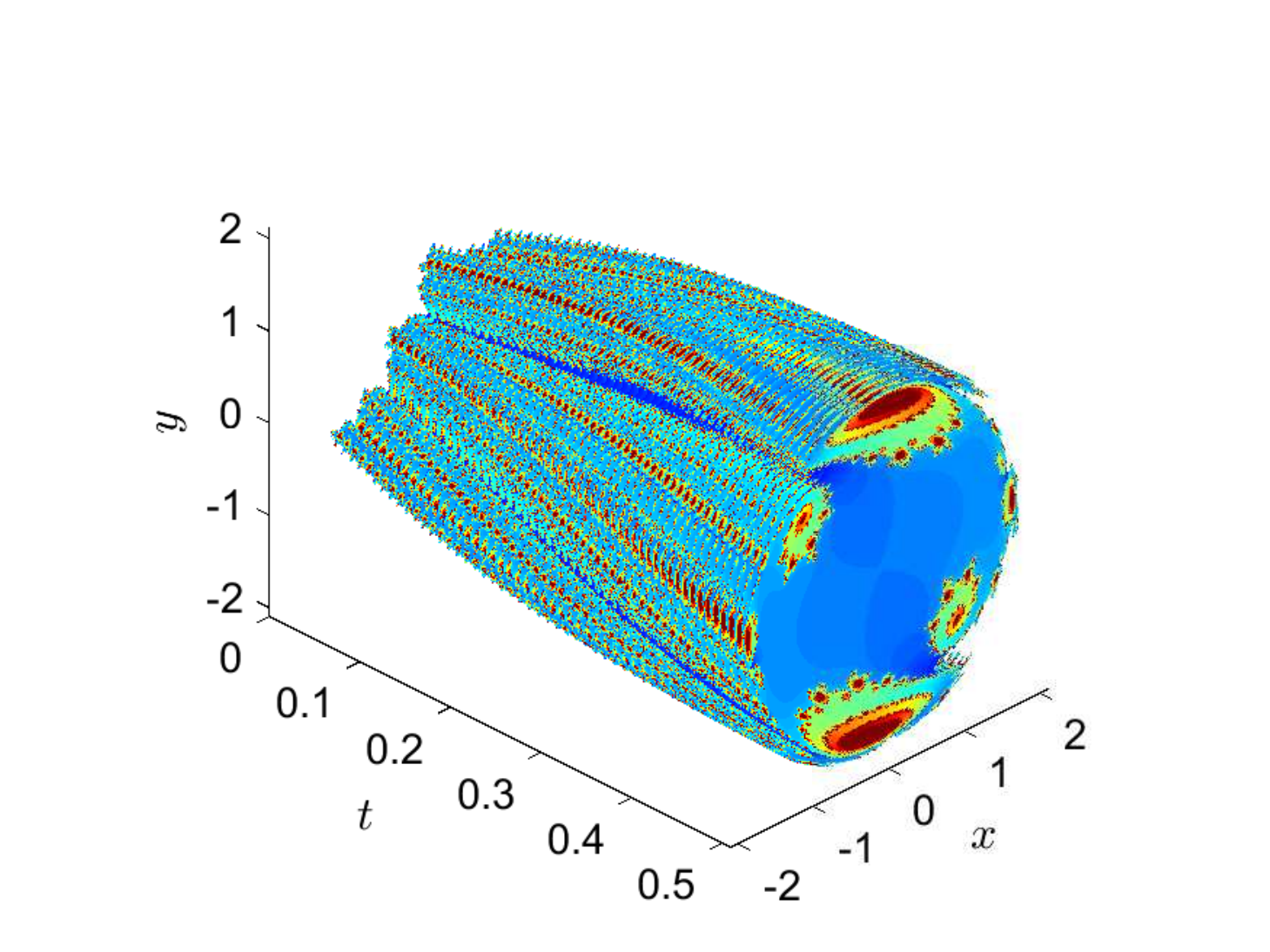}\label{ContDynAutonT}}
	\end{minipage} \hspace{0.8cm}
	\begin{minipage}{0.21\textwidth}
		\centering
		\subfigure[$ t=0 $]{\includegraphics[width = 1.5in]{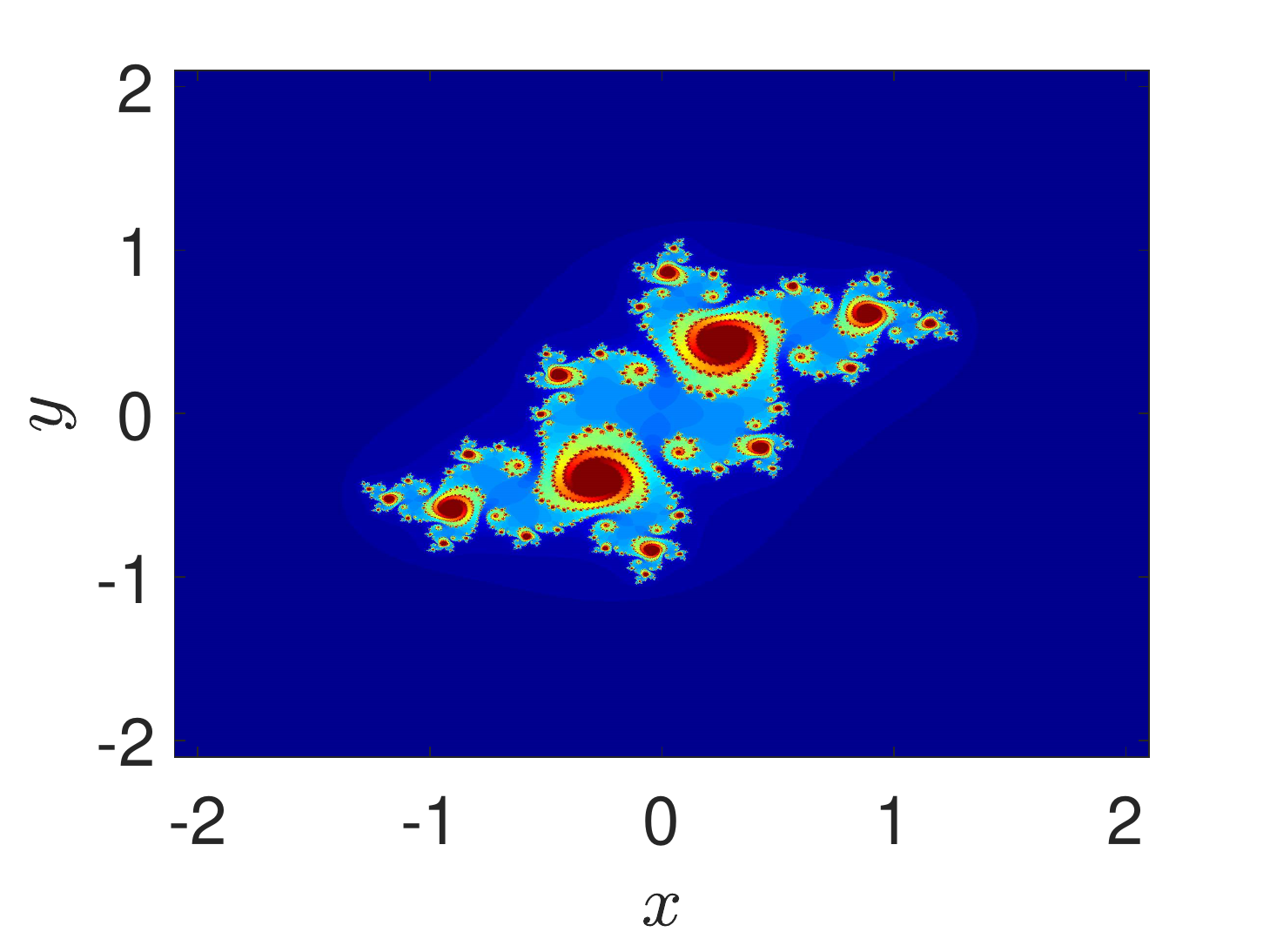}\label{ContDynAutonS0}} \\
		\subfigure[$ t=0.1 $]{\includegraphics[width = 1.5in]{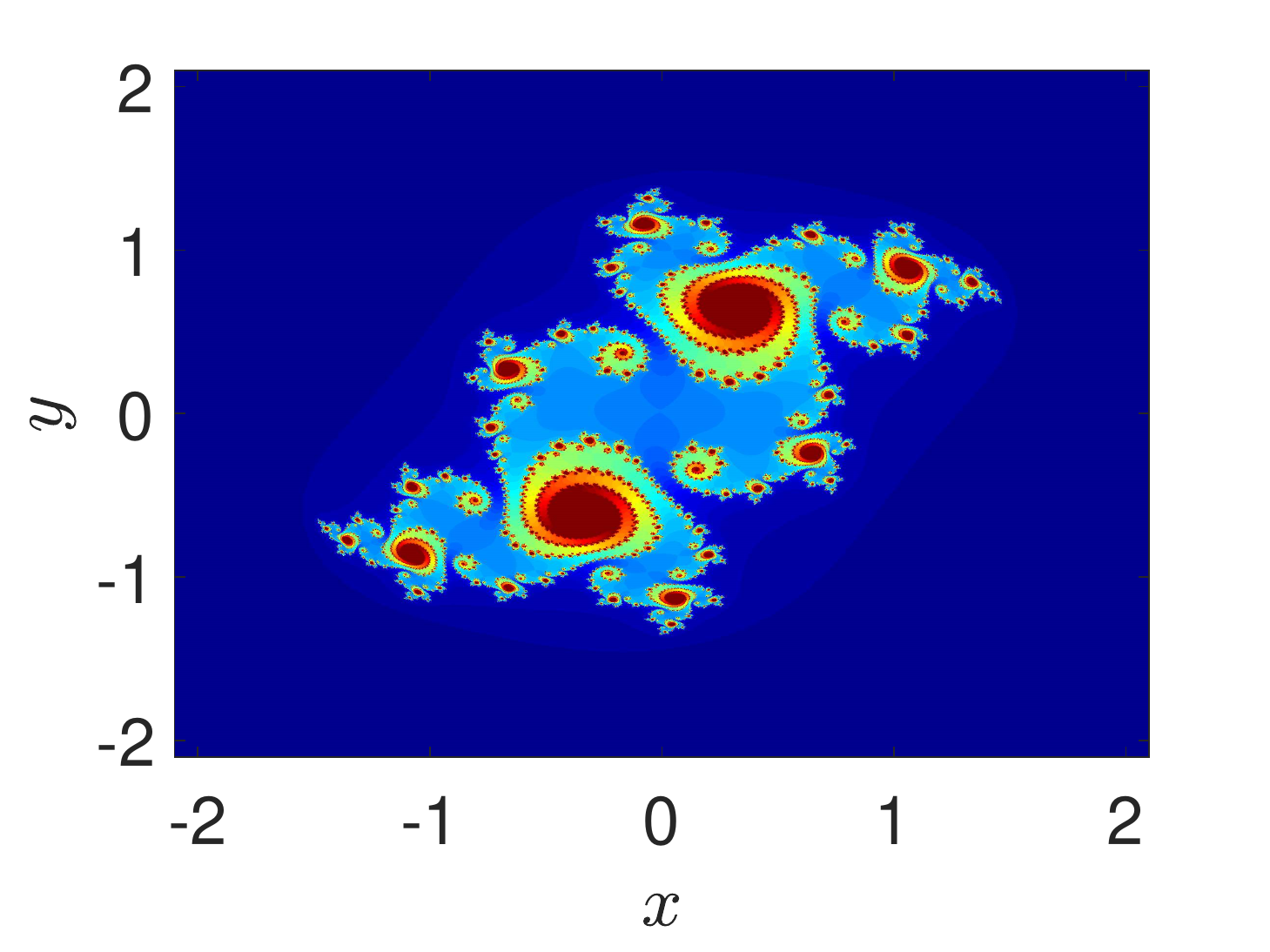}\label{ContDynAutonS1}}
	\end{minipage} \hspace{-0.4cm}
	\begin{minipage}{0.21\textwidth}
		\centering
		\subfigure[$ t=0.2 $]{\includegraphics[width = 1.5in]{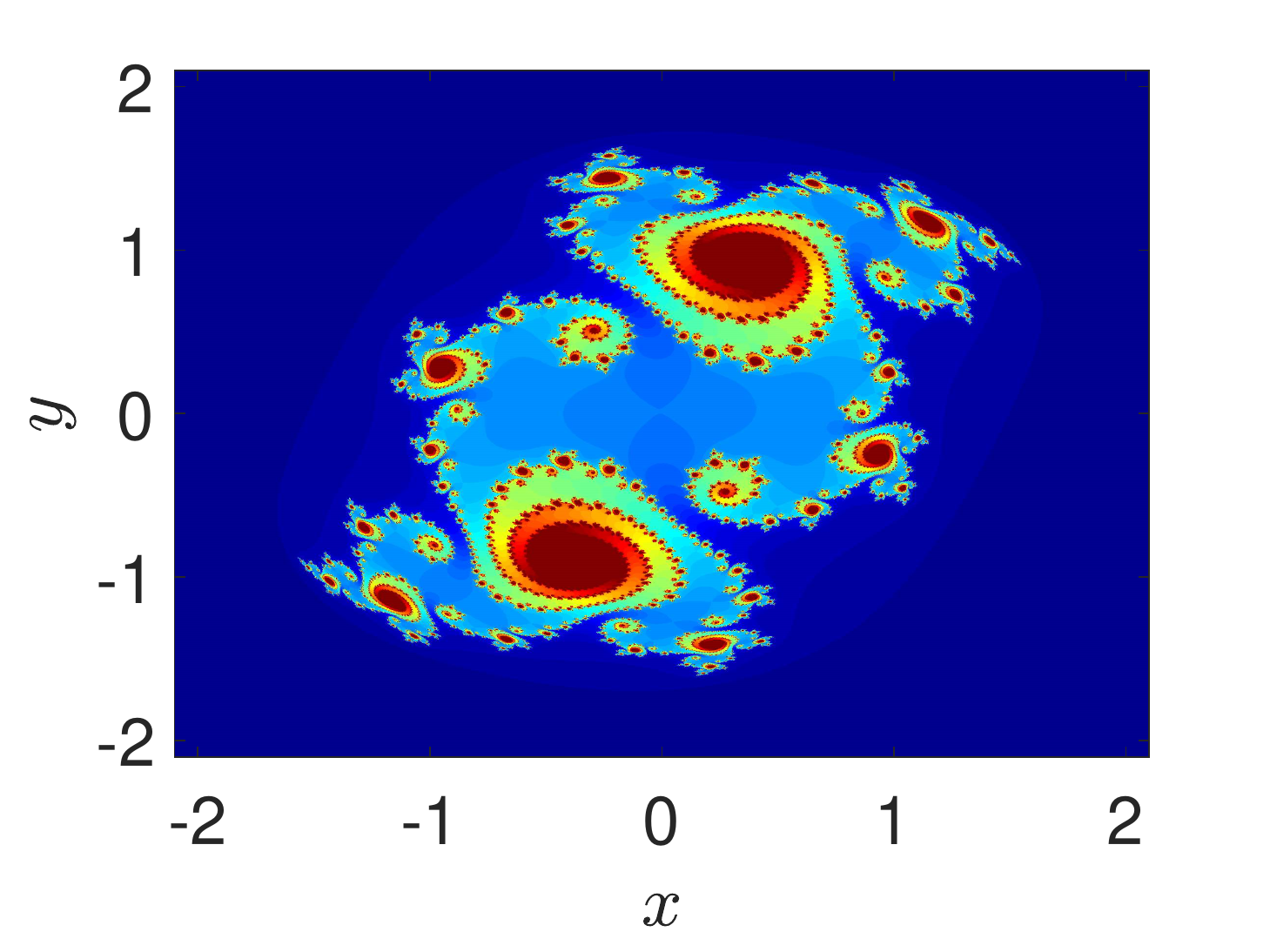}\label{ContDynAutonS2}} \\
		\subfigure[$ t=0.3 $]{\includegraphics[width = 1.5in]{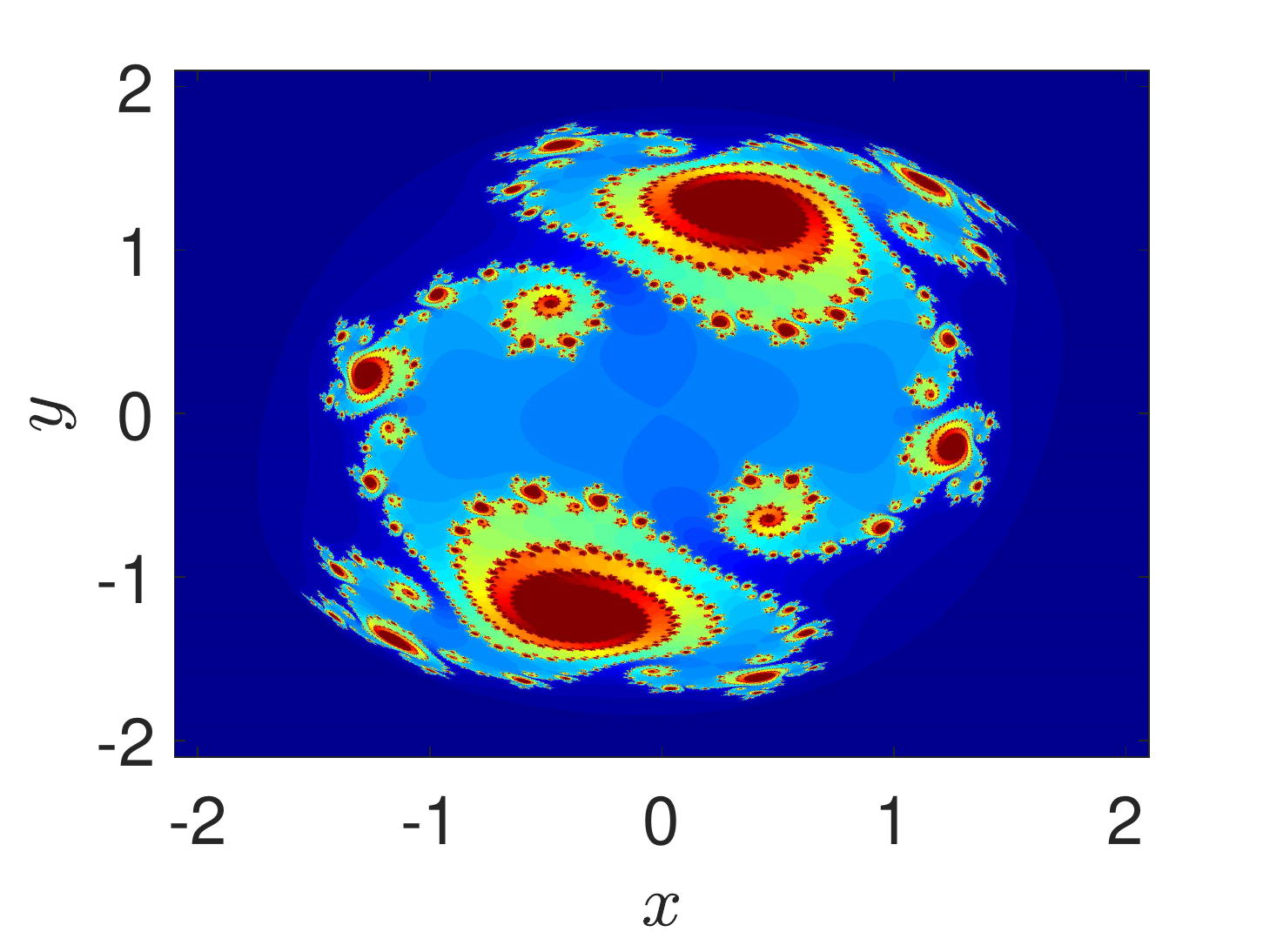}\label{ContDynAutonS3}}
	\end{minipage} \hspace{-0.4cm}
	\begin{minipage}{0.21\textwidth}
		\centering
		\subfigure[$ t=0.4 $]{\includegraphics[width = 1.5in]{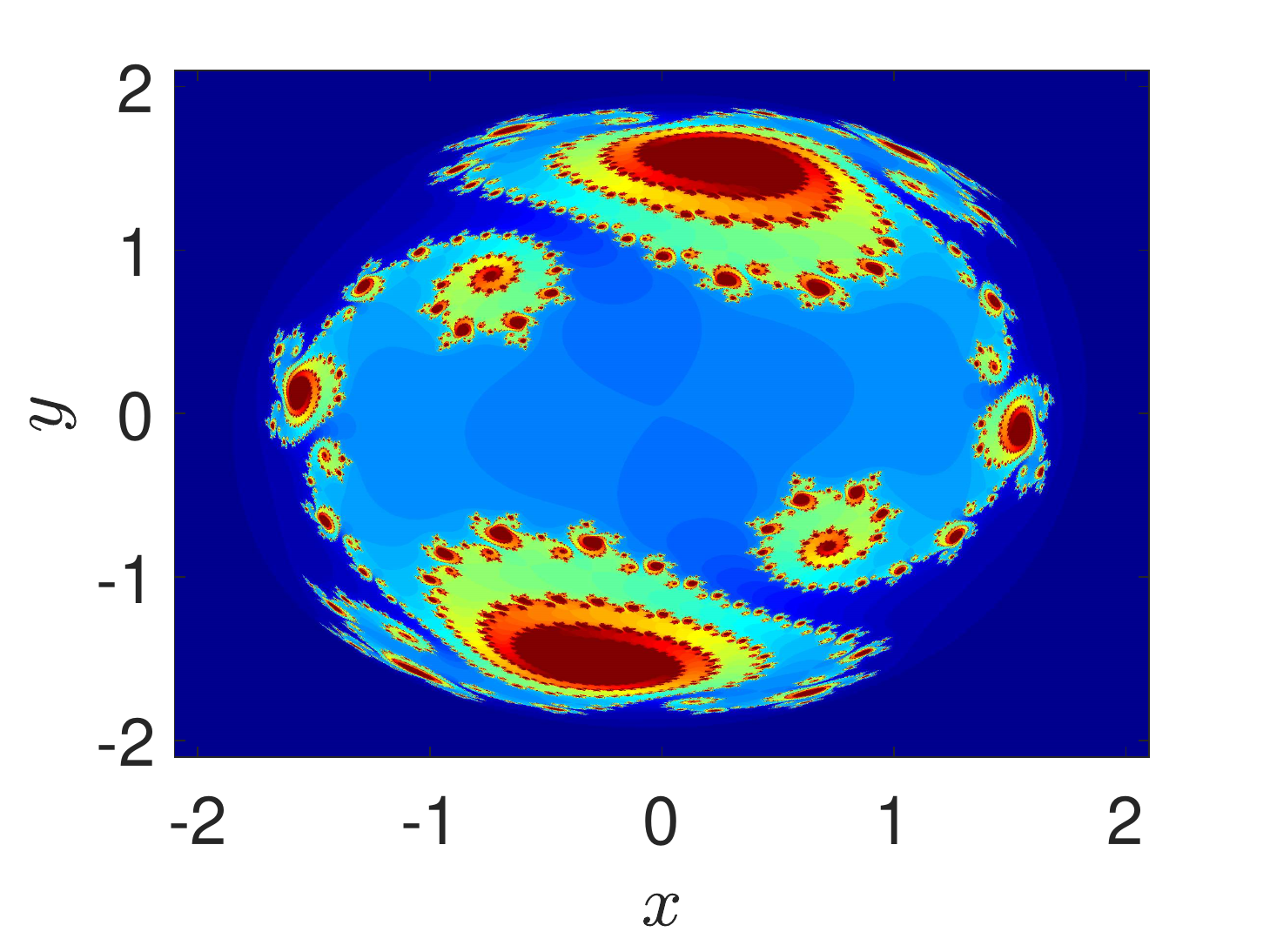}\label{ContDynAutonS4}} \\
		\subfigure[$ t=0.5 $]{\includegraphics[width = 1.5in]{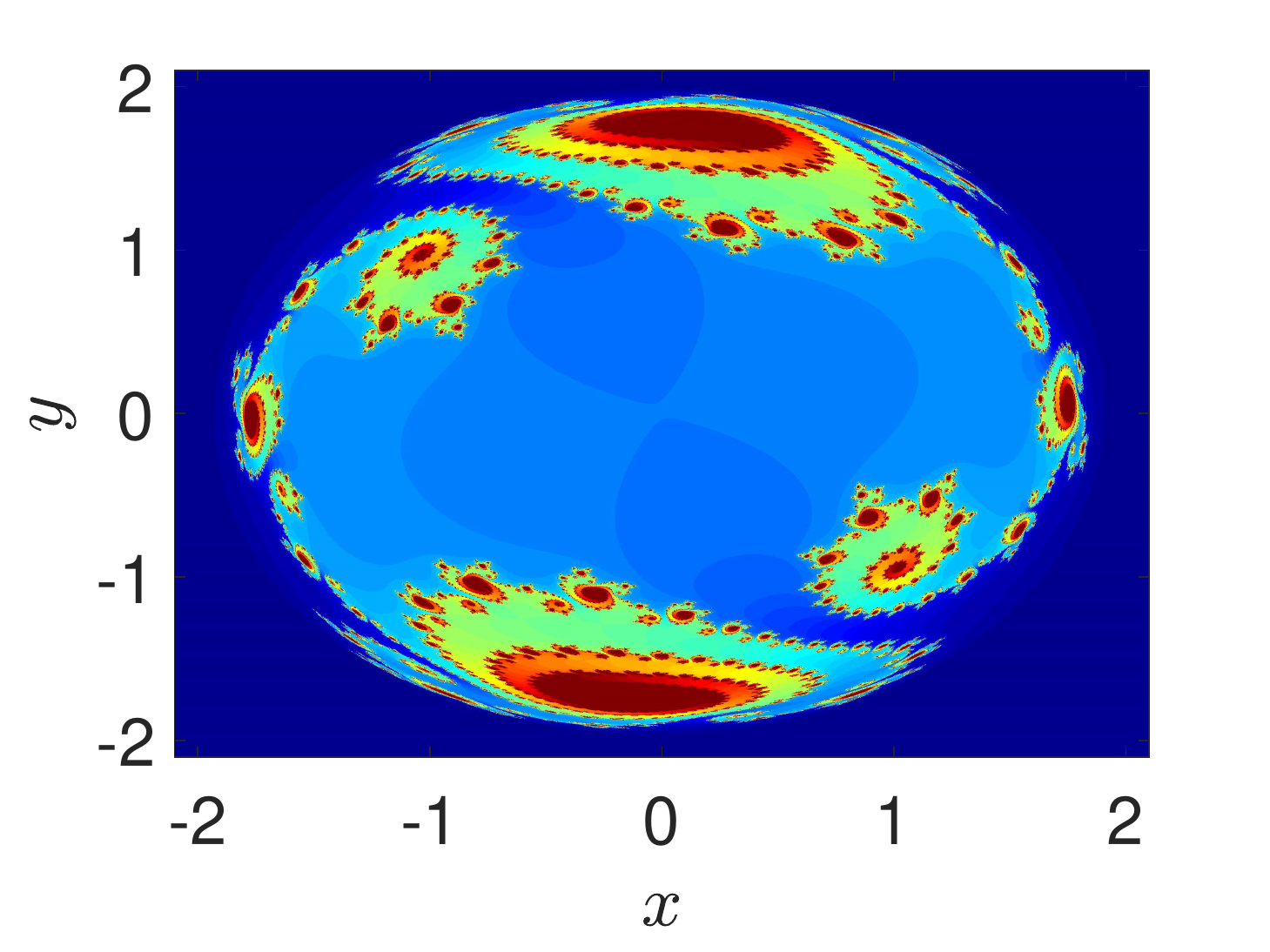}\label{ContDynAutonS5}}
	\end{minipage}}}
\caption{The fractal trajectory for (\ref{AutoODES}) and its points}
\label{ContDynAuton}   				
\end{figure}
\begin{multicols}{2}	
	 
Next, let us consider the  non-autonomous differential equation
\begin{equation} \label{ODE2}
	\frac{dz}{dt}=a z+ (\cos t+ i \sin t),
\end{equation}
and the map 
\begin{equation} \label{Dyn2}
A_t z=(z+\frac{a+i}{1+a^2}) e^{a t}-\frac{a+i}{1+a^2}(\cos t+ i \sin t),
\end{equation}
which is determined by the solutions of Eq. \ref{ODE2}. 

The map is not of a dynamical system since there is no group property for non-autonomous equations, in general. This is why, Eq. \ref{ContJul2} cannot be used for fractal mapping along the solutions of the differential equation (\ref{ODE2}). However, for the moments of time $2\pi n$, $n=1,2,\ldots,$ which are multiples of the period, the group property is valid, and therefore iterations by Eq. \ref{ContJul2} determine a fractal dynamics at the discrete moments. In the future, finding conditions to construct fractals by non-autonomous systems might be an interesting theoretical and application problem. We have applied the map  with $ a=0.01 $ and the Julia set corresponding to $ c=-0.175-0.655i$ as the initial fractal. The results of the simulation are seen in Fig. \ref{ContDynPeriod}. Since the moment $ t = \frac{\pi}{2} $ is not a multiple of the period, the section in part $(b)$ of the figure does not seem to be a fractal, but in part $(c),$ the section is a Julia set.

\section*{\normalsize \color{red}CONCLUSION}

Despite the intensive research of fractals lasts more than 35 years \textit{\cite{Mandelbrot0}}, there are still no results on mapping of the sets, and our paper is the first one to consider the problem. To say about mathematical challenges connected to our suggestions, let us start with topological equivalence of fractals and consequently, normal forms. Differential and discrete equations will be analyzed with new methods of fractal dynamics joined with dimension analysis. Next, the theory for dynamical systems which is defined as iterated maps can be developed. Therefore, mapping of fractals will be beneficial for new researches in hyperbolic dynamics, strange attractors, and ergodic theory \textit{\cite{Wiggins88,Guckenheimer80}}. The developed approach will enrich the methods for the discovery and construction of fractals in the real world and industry such as nano-fiber engineering, 3D printing, biotechnologies, and genetics \textit{\cite{Vehel,Liu,Cattani,Noorani}}.

\begin{figure}[H]
	\centering
	{\setlength{\fboxsep}{0pt}%
	\setlength{\fboxrule}{1.5pt}%
	\fcolorbox{red}{white}{
	\begin{minipage}{0.46\textwidth}
			\subfigure[$ A_t \mathcal{J} $]{\includegraphics[width = 3.2in]{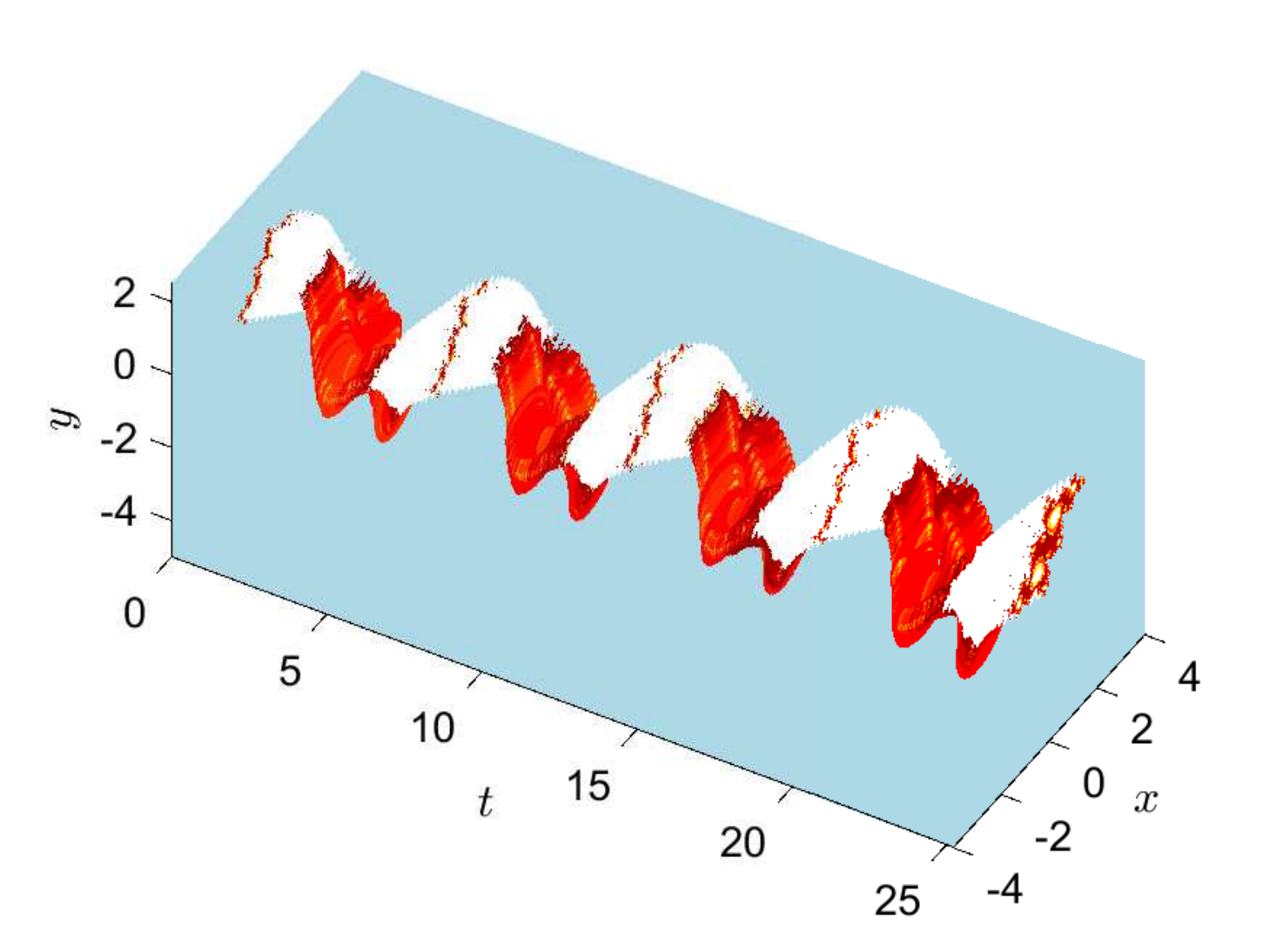}\label{ContDynPeriodT}}\\
			\subfigure[$ t=\frac{\pi}{2} $]{\includegraphics[width = 1.6in]{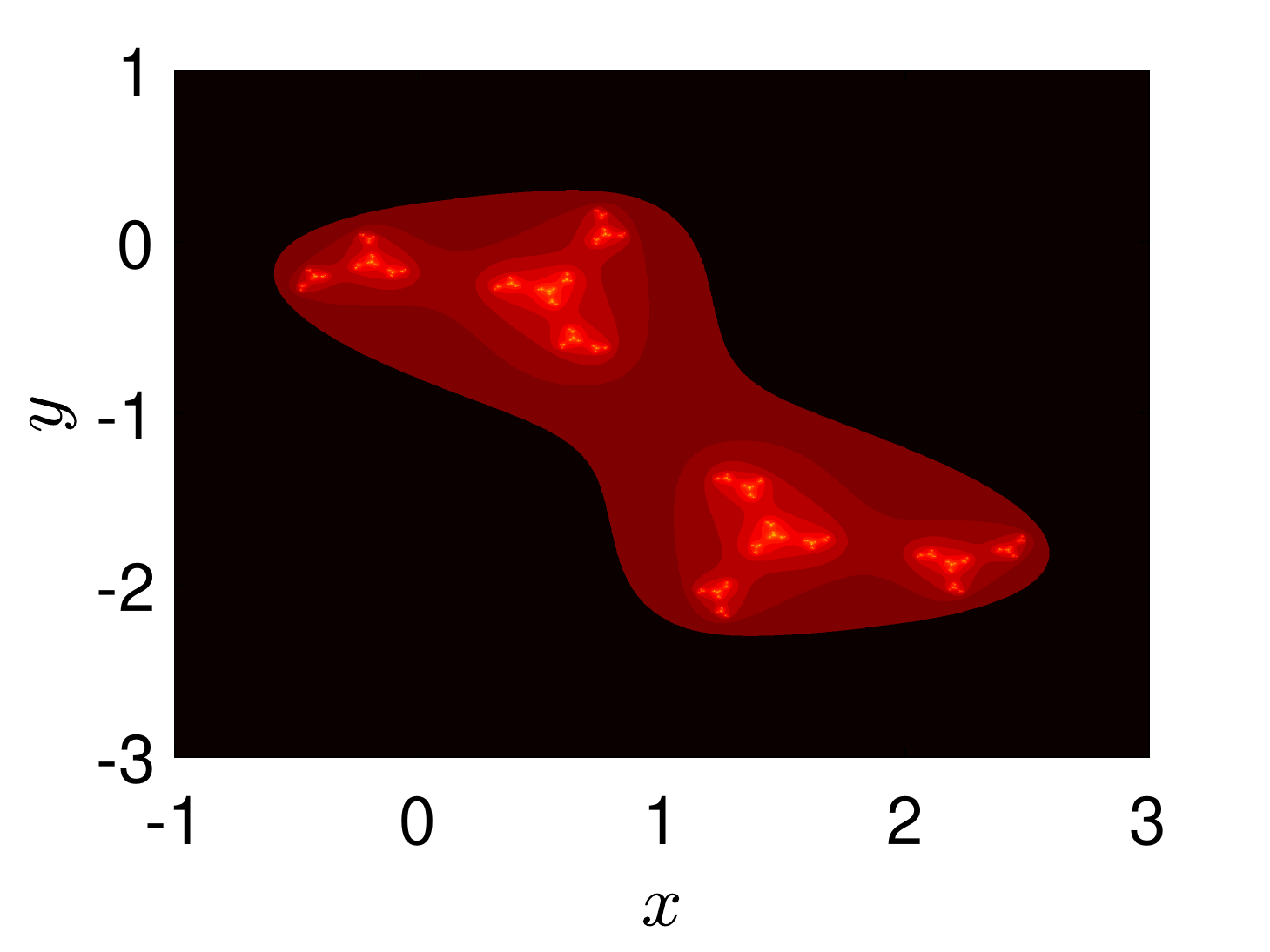}\label{ContDynPeriodS1}}
			\hspace{-0.2cm}
			\subfigure[$ t=2 \pi $]{\includegraphics[width = 1.6in]{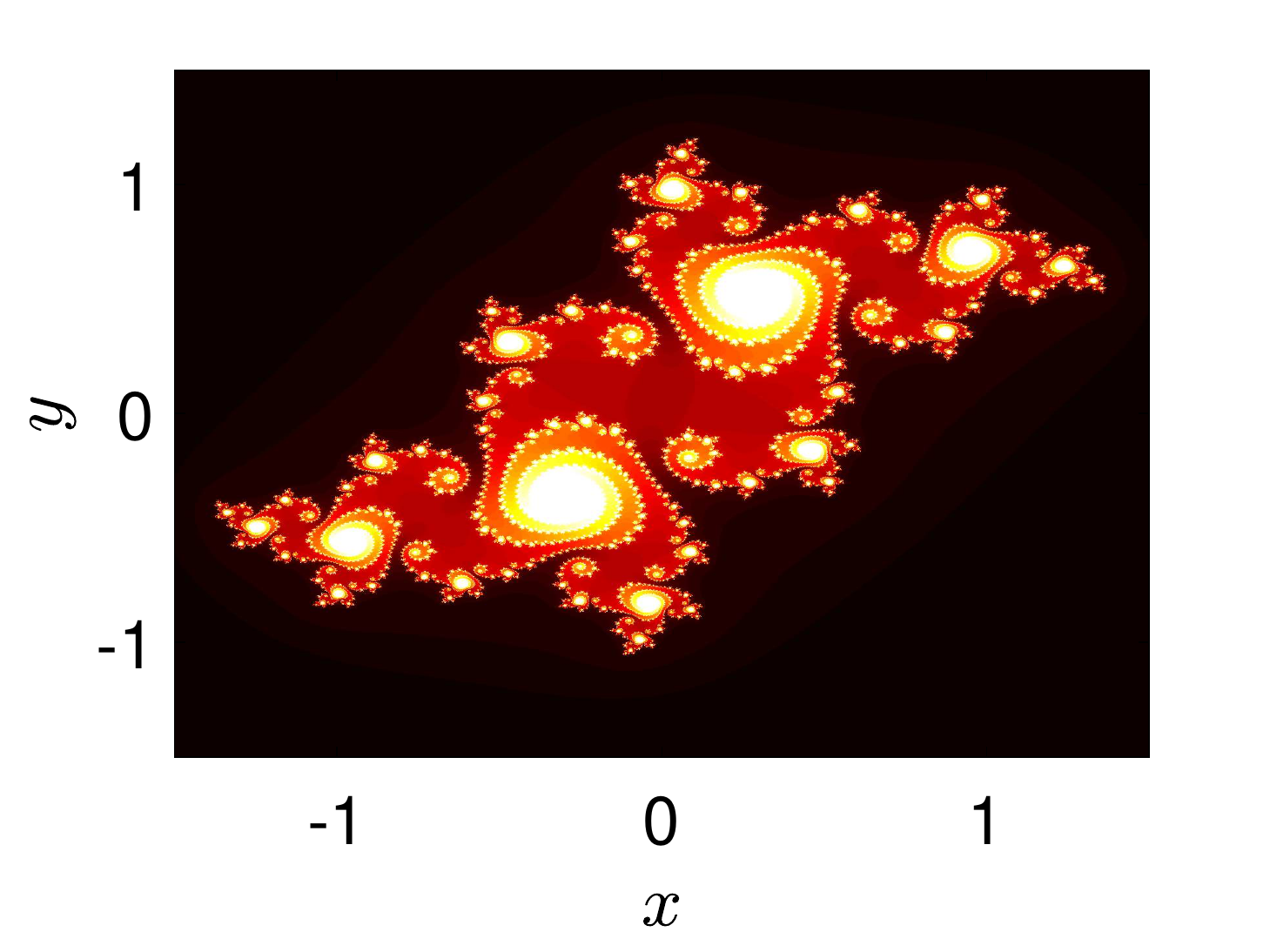}\label{ContDynPeriodS2}}
	\end{minipage}}}
	\caption{The parametric set  and its sections}
	\label{ContDynPeriod}   				
\end{figure}

\renewcommand{\refname}{\normalfont\selectfont\normalsize REFERENCES}
\color{red}

\section*{\normalsize \color{red}Acknowledgments}
\color{black}
The third author is supported by a scholarship from the Ministry of Education, Libya.

\end{multicols}	

\end{document}